\documentclass[a4paper,12pt,dvipdfmx,twoside]{article}

\title{Small eigenvalues of the rough and Hodge Laplacians under fixed volume}

\author{\Large Colette Ann\'e and Junya Takahashi}

\date{31 January 2022}
\usepackage{latexsym}
\usepackage{amsmath}
\usepackage{amsthm}
\usepackage{amssymb}
\usepackage{amscd}
\usepackage{xcolor}
\usepackage{graphicx}

\usepackage[french,english]{babel}
\usepackage{here}

\usepackage{tikz}
\usepackage{tikz-cd}
\usetikzlibrary{patterns}

\usepackage{url}

\usepackage{mathrsfs}
\usepackage{fancyhdr}
\newtheorem{thm}{Theorem}[section]

\newtheorem{lem}[thm]{Lemma}
\newtheorem{cor}[thm]{Corollary}
\newtheorem{rem}[thm]{Remark}

\renewcommand{\labelenumi}{$(\arabic{enumi})$} % (1), (2)
\numberwithin{equation}{section}

\newcommand{\rest}{{\!}\upharpoonright} % restriction

\newcommand{\dint}{\displaystyle \int}
\newcommand{\dlim}{\displaystyle \lim}
\newcommand{\dsum}{\displaystyle \sum}

\def\diam{\mathop{\mathrm{diam}}\nolimits}

\def\grad{\mathop{\mathrm{grad}}\nolimits}

\def\vol{\mathop{\mathrm{vol}}\nolimits}

\def\a{\mathop{\mathrm{\alpha}}\nolimits}
\def\b{\mathop{\mathrm{\beta}}\nolimits}
\def\e{\mathrm{\varepsilon}}
\def\vphi{\mathop{\mathrm{\varphi}}\nolimits}

\def\D{\mathop{\mathrm{\Delta}}\nolimits}
\def\N{\mathop{\mathrm{{\Bbb N}}}\nolimits}
\def\R{\mathop{\mathrm{{\Bbb R}}}\nolimits}
\renewcommand\S{\mathrm{\Bbb S}}
\def\D{\mathop{\mathrm{{\Bbb D}}}\nolimits}
\begin{document}

\maketitle

\begin{abstract}
 For each degree $p$ and each natural number $k\geq 1$, we construct on any closed manifold 
 a family of Riemannian metrics, with fixed volume such that the $k^{\text{th}}$ 
 positive eigenvalue of the rough or the Hodge Laplacian acting on differential 
 $p$-forms converge to zero. 
 In particular, on the sphere, we can choose these Riemannian metrics 
 as those of non-negative sectional curvature.
 This is a generalization of the results by Colbois and Maerten in 2010
 to the case of higher degree forms. \\

\noindent 
 {\bf R\'esum\'e.} 
 Pour chaque degr\'e $p$ et chaque entier naturel $k\geq 1$, nous construisons,
 sur toute vari\'et\'e compacte, une famille de m\'etriques riemanniennes \`a volume
 fix\'e telle que la $k^{\text{i\`eme}}$ valeur propre strictement positive du 
 Laplacien brut ou du Laplacien de Hodge agissant sur les formes diff\'erentielles 
 de degr\'e $p$ converge vers z\'ero.
 En particulier, sur la sph\`ere, nous pouvons choisir des m\'etriques \`a courbure 
 sectionnelle positive. 
 Ce r\'esultat g\'en\'eralise aux plus hauts degr\'es celui de Colbois et Maerten de 2010.
 \end{abstract}

\footnotetext{$2020$ \it{ Mathematics Subject Classification.} 
 Primary $58J50$; Secondary $35P15,$ $53C21$, $58C40$.
}

%%
%%  58J50  Spectral problems; spectral geometry; scattering theory on manifolds
%%
%%  35P15  Estimates of eigenvalues in context of PDEs
%%
%%  53C21  Methods of global Riemannian geometry, including PDE methods; 
%%     curvature restrictions
%%  53C23  Global geometric and topological methods (a la Gromov);
%%      differential geometric analysis on metric spaces
%%
%%  58C40  Spectral theory; eigenvalue problems on manifolds

%%%%%%%%%%%%%%%%%%%%%%%%%%%%%%%%%%%%%%%%%%%%%%%%%%%%%%%%%%%%%%%%%%%%%%%%%%
%%%%                  Section 1
%%%%%%%%%%%%%%%%%%%%%%%%%%%%%%%%%%%%%%%%%%%%%%%%%%%%%%%%%%%%%%%%%%%%%%%%%%

 \section{Introduction}

 We study the eigenvalue problems of two elliptic differential operators 
 acting on differential $p$-forms on a connected oriented closed Riemannian 
 manifold $(M^m,g)$ of dimension $m \ge 2$.

 One is the rough Laplacian $\overline{\Delta} = \nabla^* \nabla$,
 or the connection Laplacian, acting on $p$-forms on $(M,g)$, 
 where $\nabla$ is the covariant derivative induced from the Levi-Civita connection
 of the Riemannian metric $g$.
 The spectrum of the rough Laplacian consists only of {\bf non-negative} eigenvalues 
 with finite multiplicity.
 We denote its eigenvalues counted with multiplicity by 
%%%%%%%%
\begin{equation*}
\begin{split}
  0 \le \overline{\lambda}^{(p)}_1(M,g) \le \overline{\lambda}^{(p)}_2(M,g) \le
    \cdots  \le \overline{\lambda}^{(p)}_k(M,g) \le \cdots.
\end{split}
\end{equation*}
%%%%%%%

 The other is the Hodge-Laplacian $\Delta = d \delta + \delta d$ acting on $p$-forms 
 on $(M,g)$, where $d$ is the exterior derivative and $\delta$ its formal adjoint 
 with respect to the $L^2$-inner product.
 The spectrum of the Hodge-Laplacian consists only of non-negative eigenvalues 
 with finite multiplicity.
 We also denote its {\bf positive} eigenvalues counted with multiplicity by 
%%%%%%%%
\begin{equation*}
\begin{split}
  \underbrace{0 = \cdots = 0}_{b_p(M)} < {\lambda}^{(p)}_1(M,g) 
    \le {\lambda}^{(p)}_2(M,g) \le \cdots  \le {\lambda}^{(p)}_k(M,g)
    \le \cdots, 
\end{split}
\end{equation*}
%%%%%%%
 where the multiplicity of the eigenvalue $0$ is equal to the $p$-th Betti number
 $b_p(M)$ of $M$, by the Hodge-Kodaira-de Rham theory.
 In particular, it is independent of a choice of Riemannian metrics.

 Furthermore, since the Hodge-Laplacian $\Delta$ commutes with $d$ and $\delta$,
 we can define the $k$-th eigenvalues of the Hodge-Laplacian acting on 
 exact and co-exact $p$-forms, which are denoted by 
  $\lambda^{\prime (p)}_k (M,g)$ and $\lambda^{\prime \prime (p)}_k (M,g),$ 
 respectively.
 These are always positive.
 From the Hodge duality, it follows that for any degree $p$
%%%%%%%%
\begin{equation} \label{eq:Hodge-duality}
\begin{split}
   \lambda^{\prime (p)}_k(M,g)
     &= \lambda^{\prime \prime (m-p)}_k(M,g)
\end{split}
\end{equation}
%%%%%%%
 for $k =1,2, \dots$.
 In particular, we see that
%%%%%%%%
\begin{equation} \label{eq:first-ev-duality}
\begin{split}
 \lambda^{(p)}_1(M,g)
   &= \min \{ \lambda^{\prime (p)}_1(M,g), \ \lambda^{\prime \prime (p)}_1(M,g) \} \\
   &= \min \{ \lambda^{\prime (p)}_1(M,g), \ \lambda^{\prime (m-p)}_1(M,g) \}.
\end{split}
\end{equation}
%%%%%%%

 We are interested in the supremum and the infimum of the $k$-th eigenvalues
 under all Riemannian metrics with fixed volume on $M$. 
 Colbois and Dodziuk \cite{Colbois-Dodziuk[94]} proved that 
 there exists no universal upper bound of the first positive eigenvalue of 
 the Laplacian acting on functions under fixed volume. 
 Similar results to the rough-Laplacian acting on $p$-forms with $1 \le p \le m-1$ 
 were proved by Colbois and Maerten \cite[Theorem $1.1$]{Colbois-Maerten[10]}, 
 and to the Hodge-Laplacian acting on $p$-forms for $2 \le p \le m-2$ 
 by Gentile and Pagliara \cite{Gentile-Pagliara[95]}. 
 But, the case of $p=1, m-1$ is still unknown (cf.\ \cite{Tanno[83]}, \cite{Gentile[99]}).

 There exists no positive universal lower bound of the first positive eigenvalue of 
 the Laplacian acting on functions, if we deform a Riemannian manifold to a dumbbell 
 under fixed volume, which is called the Cheeger dumbbell \cite{Cheeger[70]}. 
 Similar results to the rough Laplacian acting on $p$-forms with $p = 0, 1, m-1, m$ 
 were also proved by Colbois and Maerten \cite[Theorem $1.2$]{Colbois-Maerten[10]}. 
 On any connected oriented closed manifold $M$ of dimension $m \ge 3,$ 
 there exists a one-parameter family of Riemannian metrics $\overline{g}_L$
 with volume one such that for $p = 0, 1, m-1, m$ and for any $k \ge 1,$  
%%%%%%%%
\begin{equation*}
\begin{split}
  \overline{\lambda}^{(p)}_k (M, \overline{g}_L) &\longrightarrow 0
   \  \text{ as } \  L \longrightarrow \infty.
\end{split}
\end{equation*}
%%%%%%%

 In the present paper, we prove similar results in the case of all degree $p$
 with $1 \le p \le m-1.$ We also prove them in the case of the Hodge-Laplacian 
 for all degree $p$ with $1 \le p \le m-1$. In the same way as Colbois and Dodziuk
 \cite{Colbois-Dodziuk[94]} who deduced their result from that on the spheres,
 for the odd dimensional spheres by Tanno \cite{Tanno[79]}, Bleecker \cite{Bleecker[83]},
 for the even dimensional spheres by Muto \cite{Muto[80]},
 using a spectral analysis of connected sums,
  we first construct such a family of Riemannian metrics with non-negative 
 sectional curvature on the $m$-dimensional standard sphere ${\Bbb S}^m$. 
 More precisely, 

%%%%%%%%%%%%%%%%%%% thm:sphere %%%%%%%%%%%%%%%%%%%%%%%%%%%%%%%%%%
\begin{thm}\label{thm:sphere}
 For $m \ge 2$ and a given degree $p$ with $1 \le p \le m-1,$ 
 there exists a one-parameter family of Riemannian metrics $\overline{g}_{p, L}$ 
 on the $m$-dimensional standard sphere ${\Bbb S}^m$  with volume one and 
 non-negative sectional curvature such that for any integer $k \ge 1$, 
%%%%%%%%%%%
\begin{enumerate}
 %% express  (i), (ii) 
 % \renewcommand{\labelenumi}{$(\roman{enumi})$}
%%%%%%%
 \item  $\ \overline{\lambda}^{(p)}_k({\Bbb S}^m, \overline{g}_{p,L}) 
           \longrightarrow  0; $     
 \item  $\ \lambda^{\prime \prime (p)}_k({\Bbb S}^m, \overline{g}_{p,L}) 
           \longrightarrow  0,$
\end{enumerate}
%%%%%%%%%
 as $L \longrightarrow \infty$.
\end{thm}
%%%%%%%%%%%%%%%%%%%%%%%%%%%%%%%%%%%%%%%

 Furthermore, we give lower bounds of the eigenvalues of the Hodge-Laplacian
 acting on $p$-forms on ${\Bbb S}^m$ 
 in Theorem \ref{thm:lower-bdd} and Corollary $\ref{cor:lower-bdd}$. 

\vspace{0.2cm}
%%%%%%%%%%%

 Next, by gluing this sphere with any closed manifold, 
 we obtain the same result as Theorem \ref{thm:sphere} for any closed manifold, 
 without keeping non-negative sectional curvature.

%%%%%%%%%%%%%%%%%    Main Theorem  %%%%%%%%%%%%%%%%%%%%%%%%%%%%%%%%%%%%
\begin{thm} \label{thm:main-thm}
 Let $M^m$ be a connected oriented closed manifold of dimension $m \ge 2.$
 For any fixed degree $p$ with $1 \le p \le m-1$, any integer $k \geq 1$ and 
 for any $\e >0$, there exists a Riemannian metric $\overline{g}_{p,\e}$ on $M$
 with volume one such that
%%%%%%%%
\begin{equation*}
\begin{split}
  0 < \overline{\lambda}^{(p)}_k (M, \overline{g}_{p,\e}) &< \e 
    \quad  \text{ and } \quad
  \lambda^{\prime \prime (p)}_k (M, \overline{g}_{p,\e}) < \e.
\end{split}
\end{equation*}
%%%%%%%
\end{thm}
%%%%%%%%%%%%%%%%%%%%%%%%%%%%%%%%%%%%%%%%%%%%%%%%%%%%%%%%%%%%%%%

 We note that Riemannian metrics $\overline{g}_{p,L}$ and $\overline{g}_{p,\e}$
 in Theorems $\ref{thm:sphere}$ and $\ref{thm:main-thm}$ depend on the degree $p$. 
 But, by taking connected sums of $M$ and $(m-1)$ spheres at distinct $(m-1)$ points, 
 we obtain a Riemannian metrics $\overline{g}_{\e}$ on $M$ with small eigenvalues 
 for all $p =1,2,\dots,m-1.$

%%%%%%%%%%%%%%%%% thm:uniformly-main %%%%%%%%%%%%%%%%
\begin{thm} \label{thm:unifomly-main}
 Let $M^m$ be a connected oriented closed manifold of dimension $m \ge 2.$
 For any $\e >0$ and any integer $k \geq 1$, there exists a Riemannian metric 
 $\overline{g}_{\e}$ on $M$ with volume one such that for any degree $p$ 
 with $1 \le p \le m-1$,
%%%%%%%%
\begin{equation*}
\begin{split}
  0 < \overline{\lambda}^{(p)}_k (M, \overline{g}_{\e}) &< \e 
    \quad  \text{ and } \quad
  \lambda^{\prime \prime (p)}_k (M, \overline{g}_{\e}) < \e.
\end{split}
\end{equation*}
%%%%%%%
\end{thm}
%%%%%%%%%%%%%

%%%%%%%%%%%%%%%%%%
\begin{rem}
\begin{enumerate}
 %% express as (i), (ii) 
 \renewcommand{\labelenumi}{$(\roman{enumi})$}
 %%%%%%%
 \item  In the case of $m=2$, Theorem $\ref{thm:main-thm}$ is covered with the result 
 by Colbois and Maerten \cite{Colbois-Maerten[10]}. 
 \item  The same results for the Hodge-Laplacian were obtained from the
  results by Guerini \cite{Guerini[04]} and Jammes \cite{Jammes[08]}, \cite{Jammes[11]}.
  Although the sectional curvature for their Riemannian metrics on ${\Bbb S}^m$ 
  diverges to $- \infty$, our Riemannian metric constructed in 
  Theorem $\ref{thm:sphere}$ has non-negative sectional curvature. 
  This is our advantage.
\item  As a consequence of our result on spheres, Theorem $\ref{thm:sphere}$, 
  their exists no lower bound of the positive eigenvalue of degree $p$ with 
  $1 \leq p \leq m-1$ depending only on the dimension, the volume and a lower bound
  of the sectional curvature. See Remark $\ref{conj:lott}$ below.
%%%%
 \item  Jammes \cite{Jammes[08]} constructed similar Riemannian metrics within
 a fixed conformal class for $m \ge 5$, except for $p= \frac{m}{2}$ if $m$ is even.
 % (cf.\ surgery by Ammann, Dahn, Humbert \cite{Ammann-Dahl-Humbert[09]}).
\end{enumerate}
\end{rem}
%%%%%%%%%%%%%%%%%

 The present paper is organized as follows:
 In Section $2$, we recall the Weizenb\"ock formula and the properties 
 of parallel forms.
 In Section $3$, we consider the case of the sphere, 
 and give the proof of Theorem \ref{thm:sphere}.
 In Section $4$, we give lower bounds for the eigenvalues of 
 the Hodge-Laplacian on the sphere.
 In Section $5$, we consider the case of a general manifold, 
 and give the proof of Theorems \ref{thm:main-thm} and \ref{thm:unifomly-main}.
 In Section $6$, as an appendix, we prove the convergence theorem
 of the eigenvalues of the rough Laplacian acting on $p$-forms, 
 when one side of a connected sum of two closed Riemannian manifolds
 collapses to a point.

 \vspace{0.5cm}

%% 
%%  Acknowledgement
%%
\noindent
 {\bf Acknowledgement.}
 We thank the referee for the interest on our work, valuable improvement and the idea of 
 Theorem \ref{thm:unifomly-main}.
 The second named author is supported by the Grants-in-Aid for Scientific Research (C),
 Japan Society for the Promotion of Science, No.\ 16K05117.

 % JSPS KAKENHI Grant Number $16K05117.$

%%%%%%%%%%%%%%%%%%%%%%%%%%%%%%%%%%%%%%%%%%%%%%%%%%%%%%%%%%%%%%%%%%%%%%%%%
%%%%%%%%               SECTION      2
%%%%%%%%%%%%%%%%%%%%%%%%%%%%%%%%%%%%%%%%%%%%%%%%%%%%%%%%%%%%%%%%%%%%%%%%%

\section{Notations and basic facts}

 We fix the notations used in the present paper.
 Let $(M^m,g)$ be a connected oriented closed Riemannian manifold of dimension $m \ge 2.$
 The metric $g$ defines a volume element $d \mu_g$ and a scalar product on the fibers of
 any tensor bundle. The $L^2$-inner product of the space of all smooth $p$-forms $\Omega^p(M)$ 
 is defined as, for any $p$-form $\vphi, \psi$ on $M$
%%%%%%%%
\begin{equation*} \label{eq:L^2-metric}
\begin{split}
  ( \vphi, \psi )_{L^2(M,g)} :&= \dint_M \langle \vphi, \psi \rangle d \mu_g
   \quad  \text{ and }  \quad
  \| \vphi \|^2 _{L^2(M,g)} := ( \vphi, \vphi )_{L^2(M,g)}.
\end{split}
\end{equation*}
%%%%%%%
 The space of $L^2$ $p$-forms $L^2 (\Lambda^p M,g)$ is the completion of 
 $\Omega^p (M)$ with respect to this $L^2$-norm.

  For a positive constant $a >0$, it is easy to see that 
%%%%%%%%
\begin{equation} \label{eq:scaling-property}
\begin{split}
  \overline{\lambda}^{(p)}_k (M, a g)
      &= a^{-1} \overline{\lambda}^{(p)}_k(M, g),  \quad
  \lambda^{(p)}_k (M, a g) = a^{-1} \lambda^{(p)}_k (M, g),  \\
  \vol (M, ag) &=  a^{\frac{m}{2}} \vol (M, g),
\end{split}
\end{equation}
%%%%%%%
 where $\vol (M,g)$ denotes the volume of $(M,g).$
 Thus, if we take a new Riemannian metric 
%%%%%%%%
\begin{equation} \label{eq:normalized-metric}
\begin{split}
   \overline{g} &:= \vol (M,g)^{- \frac{2}{m}} \, g,
\end{split}
\end{equation}
%%%%%%%
 then the volume is one: $\vol (M, \overline{g}) = 1.$
 Therefore, instead of considering a volume normalized metric, 
 we may consider the following invariants
%%%%%%%%
\begin{equation*}
\begin{split}
  \overline{\lambda}^{(p)}_k (M,g) \vol(M,g)^{\frac{2}{m}},  \quad
  \lambda^{(p)}_k (M,g) \vol(M,g)^{\frac{2}{m}}.
\end{split}
\end{equation*}
%%%%%%%

The relation between the rough  and Hodge Laplacians  on $p$-forms
 is given by the Weitzenb\"ock formula: for any $p$-form $\vphi$ on $M$,
%%%%%%%%
\begin{equation} \label{eq:Weitzenbock-formula}
\begin{split}
   \Delta \vphi &= \overline{\Delta} \vphi + F_p (\vphi), \\
\end{split}
\end{equation}
%%%%%%%
 where $F_p$ is the Weitzenb\"ock curvature tensor defined as 
%%%%%%%%
\begin{equation} \label{eq:Weitzenbock-curvature}
\begin{split}
   F_p (\vphi) &= - \dsum^m_{i,j=1} e^i \wedge i_{e_j} ( R(e_i,e_j) \vphi),
\end{split}
\end{equation}
%%%%%%%
 where $R$ denotes the curvature tensor with respect to the covariant derivative
 induced from the Levi-Civita connection and $i_X$ denotes the interior product of 
 a vector $X$, and $\{ e_1, \dots, e_m \}$ is a local orthonormal frame and
 $\{ e^1, \dots, e^m \}$ is its dual frame.

 Now, after taking the scalar product of $\vphi$ with \eqref{eq:Weitzenbock-formula}, 
 we obtain the Bochner formula: for each point on $M$, 
%%%%%%%%
\begin{equation} \label{eq:pointwise-Weitzenbock-formula}
\begin{split}
 \dfrac{1}{2} \Delta (|\vphi|^2)
  &= - | \nabla \vphi |^2 + \langle \overline{\Delta} \vphi, \vphi \rangle\\
  &= \langle \Delta \vphi, \vphi \rangle
   - | \nabla \vphi |^2 - \langle F_p (\vphi), \vphi \rangle.
\end{split}
\end{equation}
%%%%%%%

 By the Hodge-Kodaira-de Rham theory, the kernel of the Hodge-Laplacian 
acting on $p$-forms consists of harmonic $p$-forms, whose dimension is 
equal to the $p$-th Betti number of $M$, that is, a topological invariant of $M$.
 In contract, the kernel of the rough Laplacian acting on $p$-forms 
 consists of parallel $p$-forms, whose dimension is not a topological invariant.
 In fact, we can kill all parallel $p$-forms  under local perturbation 
 of a Riemannian metric.

%%%%%%%%%%%%%%%%%%%%%%%
\begin{lem} \label{lem:vanishing-parallel}
 Let $(M,g)$ be a connected oriented closed Riemannian manifold.
 If $F_p$ is positive definite at one point, there exist
 no non-zero parallel $p$-forms on $(M,g)$.
\end{lem}
%%%%%%%%%%%%%%%%%%%%%%%%

%%%%%%%%%%%%%
\begin{proof}
 We prove this by contradiction.
 Let $\vphi$ be a non-zero parallel $p$-form on $M$.
 Then, $\vphi$ is harmonic and of constant norm.
 By the assumption, there exists an open subset $U$ of $M$ such that
  $\langle F_p (\vphi), \vphi \rangle > 0$ on $U$. 
 From the Bochner formula on $U$
%%%%%%%%
\begin{equation*} \eqref{eq:pointwise-Weitzenbock-formula}
\begin{split}
 \dfrac{1}{2} \Delta (|\vphi|^2) &= \langle \Delta \vphi, \vphi \rangle
   - | \nabla \vphi |^2 - \langle F_p (\vphi), \vphi \rangle,
\end{split}
\end{equation*}
%%%%%%%
 we have 
%%%%%%%%
\begin{equation*}
\begin{split}
  0 = | \nabla \vphi |^2  &= - \langle F_p (\vphi), \vphi \rangle < 0,
\end{split}
\end{equation*}
%%%%%%%
 which is a contradiction.
\end{proof}
%%%%%%%%%%%%%%%%%%%%%%%%%%%%%%%%%%%

%%%%%%%%%%%%%%%%%%%%%%%
\begin{lem} \label{lem:local-perturbations}
  Let $(M^m,g)$ be a connected oriented closed Riemannian manifold.
 For any open subset $U$, there exists a Riemannian metric $g^{\prime}$ on $M$
 with $g^{\prime} = g$ on $M \setminus U$ such that 
 all parallel $p$-forms with respect to $g^{\prime}$ are zero.
\end{lem}
%%%%%%%%%%%%%%%%%%%%%%%%

%%%%%%%%%%%%%
\begin{proof}
 We take any point $x_0$ in any open subset $U$ of $M$.
 On a neighborhood of $x_0$, we deform the Riemannian metric $g$ to $g^{\prime}$ 
 such that $g^{\prime}$ has constant sectional curvature $1$.
 The curvature operator is also $1$ on this neighborhood of $x_0$
 (see \cite{Petersen[16]}, p.84, Proposition $3.1.3$).
 Since the Weitzenb\"ock curvature tensor $F_p$ is controlled below 
 by the curvature operator (see \cite{Gallot-Meyer[75]} p.264, Corollary $2.6$),
 we see that $F_p \ge p(m-p) > 0$ at $x_0$.
 Hence, from Lemma \ref{lem:vanishing-parallel}, we see that
 $(M, g^{\prime})$ has no non-zero parallel $p$-forms.
\end{proof}
%%%%%%%%%%%%%%%%%%%%%%%%%%%%%%%%%%%

%%%%%%%%%%%%%%%%%%%%%%%%%%%%%%%%%%%%%%%%%%%%%%%%%%%%%%%%%%%%%%%%%%%%%%%%%
%%%%%%%%               SECTION  
%%%%%%%%%%%%%%%%%%%%%%%%%%%%%%%%%%%%%%%%%%%%%%%%%%%%%%%%%%%%%%%%%%%%%%%%%

\section{Small eigenvalues on the sphere ${\Bbb S}^m$}

We first consider the case of the $m$-dimensional standard sphere ${\Bbb S}^m$.

\noindent
 {\it Notations.}
 For a dimension $n$, Let $g_{{\S}^{n}}$ be the Riemannian metric on ${\Bbb S}^n$ 
 of constant sectional curvature one.
 We denote by $\D^n$ the $n$-dimensional closed disk, and let $g_{{\D}^{n}}$
  a fixed Riemannian metric on it, 
 which is identified with $[0,2] \times {\S}^{n-1}$, 
 of non-negative sectional curvature $K_{g_{{\D}^{n}}} \ge 0$. 
 We can, in addition, assume that $g_{{\D}^{n}}$ is a product metric near the boundary.
 In fact, if we take a smooth positive function $f(r)$ on the interval $[0,2]$ 
 satisfying that 
%%%%%%%%%%%%%
\begin{equation*}
\begin{split}
  f(r) &= 
 \begin{cases}
   \sin (r) &  \text{ on } [0,1], \\
   \ 1      &  \text{ on } [3/2,2],
 \end{cases}
\end{split}
\end{equation*}
%%%%%%%%%%%%
 (Note that $\sin(1) \sim 0.84$.) 
 and $0 \le f^{\prime}(r) \le 1$ and $f^{\prime \prime}(r) \le 0,$
 then the metric $g_{{\D}^{n}}$ is written as 
%%%%%%%%%%%%%
\begin{equation}
\begin{split}
  g_{{\D}^{n}} &= dr^2 \oplus f^2(r) g_{\S^{n-1}} 
    \  \text{ on } [0,2] \times \S^{n-1}.
\end{split}
\end{equation}
%%%%%%%%%%%%
 The sectional curvatures of $g_{{\D}^{n}}$ are given, 
 if $X$ and $Y$ are orthonormal vectors tangent to the angle directions, by 
%%%%%%%%%%%%%
\begin{equation*}
\begin{split}
  % \text{the radial direction:}\
  K(\partial_r,X) = - \dfrac{f^{\prime \prime}(r) }{f(r)}, \quad
  % \text{the angle direction:} \
  K(X,Y)=\dfrac{ 1 - ( f^{\prime}(r))^2 }{f^2(r)},
\end{split}
\end{equation*}
%%%%%%%%%%%%
 both of which are non-negative (e.g., Petersen \cite{Petersen[16]}, {\bf 4.2.3}, p.121).
%%%%%%%%%%%%%%%%
\begin{proof}[{\it Proof of Theorem \ref{thm:sphere}.}]
 We take any degree $p$ with $1 \le p \le m-1.$
 We consider the decomposition (see Figure 1) %(\ref{fig:sphere-docomposition})
%%%%%%%%
\begin{equation} \label{eq:sphere-decomp}
\begin{split}
 {\Bbb S}^m &= \big( {\Bbb S}^{p} \times {\Bbb D}^{m-p}_{L} \big) 
        \cup_{ {\Bbb S}^p \times {\Bbb S}^{m-p-1} } 
     \big( {\Bbb D}^{p+1} \times {\Bbb S}^{m-p-1} \big), 
\end{split}
\end{equation}
%%%%%%%
 and set 
%%%%%%%%
\begin{equation*}
\begin{split}
  H_1 &:=  \S^{p} \times \D^{m-p}_{L}  \   \text{ and } \
  H_2 :=  \D^{p+1} \times \S^{m-p-1}.
\end{split}
\end{equation*}
%%%%%%%%%%%%%%%%%%%%

%%%%%%%%%%%%%%%%%%%%%%%%%%%%%%%
\begin{figure}[H] \label{fig:sphere-docomposition} % [h]
\begin{center}
\begin{tabular}{c}
%%%%  fig:1
\begin{minipage}{0.5\hsize}
 \begin{center}
%%%%%%%%%%%%%%%%
 \begin{tikzpicture} [shift={(0,-1)}, scale=0.5]
  % torus
  \draw (-3.5,0) .. controls (-3.5,2) 
    and (-1.5,2.5) .. (0,2.5);
  \draw [xscale=-1] (-3.5,0) .. controls (-3.5,2) 
    and (-1.5,2.5) .. (0,2.5);
  \draw [rotate=180] (-3.5,0) .. controls (-3.5,2) 
    and (-1.5,2.5) .. (0,2.5);
  \draw [yscale=-1] (-3.5,0) .. controls (-3.5,2) 
    and (-1.5,2.5) .. (0,2.5);
  \draw (-2,0.3) .. controls (-1.5,-0.3) 
    and (-1,-0.5) .. (0,-0.5) .. controls (1,-0.5) 
    and (1.5,-0.3) .. (2,0.3);
  \draw (-1.75,0) .. controls (-1.5,0.3) 
    and (-1,0.5) .. (0,0.5) .. controls (1,0.5) 
    and (1.5,0.3) .. (1.75,0);
  %  disk
  \draw [pattern=north west lines]
   (0,-2.5) .. controls (0.5,-2.0)
    and  (0.5,-1.0) .. (0,-0.5);
  \draw [dashed, pattern=north west lines]
   (0,-2.5) .. controls (-0.5,-2.0)
    and  (-0.5,-1.0) .. (0,-0.5);
  \node at (0.2,-3.5) {$\D^{m-p}_{L}$}; 
 \end{tikzpicture}
 \end{center}
\end{minipage}
%%%%%%%%%%%%%%%%%%%%%%%%
 \hspace{-1.5cm} $\bigcup_{{\Bbb S}^p \times {\Bbb S}^{m-p-1}}$ 
 \hspace{-1.5cm} 
%%%%%%%%%%%%%%%%%%%%%%%%%%%%%%%
%%%%  fig:2
\begin{minipage}{0.5\hsize}
 \begin{center}
 \begin{tikzpicture}[shift={(0,-1)}, scale=0.5]
 % torus
  \draw (0,-3.5) .. controls (-2,-3.5) 
    and (-2.5,-1.5) .. (-2.5,0);
  \draw [xscale=-1] (0,-3.5) .. controls (-2,-3.5) 
    and (-2.5,-1.5) .. (-2.5,0);
  \draw [rotate=180] (0,-3.5) .. controls (-2,-3.5) 
    and (-2.5,-1.5) .. (-2.5,0);
  \draw [yscale=-1] (0,-3.5) .. controls (-2,-3.5) 
    and (-2.5,-1.5) .. (-2.5,0);
  \draw (-0.3,-2) .. controls (0.3,-1.5) 
    and (0.5,-1) .. (0.5,0) .. controls (0.5,1) 
    and (0.3,1.5) .. (-0.3,2);
  \draw (0,-1.75) .. controls (-0.3,-1.5) 
    and (-0.5,-1) .. (-0.5,0) .. controls (-0.5,1) 
    and (-0.3,1.5) .. (0,1.75);
  %  disk
  \draw [dashed, pattern=north west lines]
   (2.5,0) .. controls (2.0,0.5)
    and  (1.0,0.5) .. (0.5,0);
  \draw [pattern=north west lines]
   (2.5,0) .. controls (2.0,-0.5)
    and  (1.0,-0.5) .. (0.5,0);
  \node at (3.8,0) {$\D^{p+1}$}; 
  \node at (0.2,-4.3) {$\ $};
%%%%%
\end{tikzpicture}
\end{center}
\end{minipage}
%%%%%%%%%%
\end{tabular}
 \end{center}
%%%%%%%%%%%%%%
 \caption{${\Bbb S}^m \cong \big( {\Bbb S}^p \times {\Bbb D}^{m-p}_L \big) 
   \cup_{{\Bbb S}^p \times {\Bbb S}^{m-p-1}} 
      \big( {\Bbb D}^{p+1} \times {\Bbb S}^{m-p-1} \big)$}
\end{figure}
%%%%%%%%%%%%%%%%%%%%%%%%%%%%%%%%%%%%%

 For any real number $L >0,$ we construct a one-parameter family of Riemannian 
 metrics $g_{p,L}$ on ${\Bbb S}^m$. First we introduce a one-parameter family of
 Riemannian metrics $g_{{\D}^{m-p},L}$ on 
 $\D^{m-p}$ containing a long cylinder as follows (see Figure $2$): 
  % \ref{fig:long-cylinder}):
%%%%%%%%%%%%%
\begin{equation*}
\begin{split}
  g_{{\D}^{m-p},L} &:= 
 \begin{cases}
   dr^2 \oplus f^2(r) g_{\S^{m-p-1}} &  \text{on } [0,2] \times \S^{m-p-1}, \\
   dr^2 \oplus g_{\S^{m-p-1}}  &  \text{on } [2,L+2] \times \S^{m-p-1}.
 \end{cases}
\end{split}
\end{equation*}
%%%%%%%%%%%%

%%%%%%%%%%%%%%%%%%%%%%%%%%%%%%%
\begin{figure}[H] \label{fig:long-cylinder}
\begin{center}
\begin{tikzpicture} % [scale=0.7]
 % \draw [help lines] (-2.0,-1.0) grid (6.0,2.0);
 % cylinder 
  \draw (-0.5,1) arc [radius=0.5, start angle=90, end angle=270]; % disk
  \draw[draw] (-0.5,1) -- (4.0,1);   %  cylinder
  \draw[draw] (-0.5,0) -- (4.0,0); %
%%%%
  \node at (2.0,1.5) {${\Bbb D}^{m-p}_{L}$};
  \draw (-1,0) node[below] {$0$}; % 0
  \draw (0,0) node[below] {$2$}; % 1
 % \draw (1,0) node[below] {$2$}; % 2
  \draw (4.2,0) node[below] {$L+2$}; % L+2
  \draw[<->] (0.4,-0.3) -- (3.5,-0.3);
  \node at (2,-0.6) {$L$};
%%%%
  \draw (0,0) .. controls (0.3,0.2) 
    and (0.3,0.8) .. (0,1.0);   % B\'ezier curve
  \draw [dotted] (0,0) .. controls (-0.3,0.3) 
    and (-0.3,0.8) .. (0,1.0);
%%%
  \draw (4,0) .. controls (4.3,0.3) 
    and (4.3,0.8) .. (4,1.0);
  \draw [dotted] (4,0) .. controls (3.7,0.3) 
    and (3.7,0.8) .. (4,1.0);
%%%%%%%%%%%%
\end{tikzpicture}
   \vspace{-0.5cm}
\end{center}
   \caption{long disk ${\Bbb D}^{m-p}_{L}$}
\end{figure}
%%%%%%%%%%%%%%%%%%%%%%

 Then, we define the smooth Riemannian metric $g_{p,L}$ on ${\Bbb S}^m$ as 
%%%%%%%%%%%%%
\begin{equation} \label{eq:metric-g_L}
\begin{split}
  g_{p,L} &:= 
 \begin{cases}
   g_{{\Bbb S}^p} \oplus g_{\D^{m-p},L}  &  \text{ on } 
        H_1 = {\Bbb S}^{p} \times \D^{m-p}_{L}, \\
   g_{\D^{p+1}} \oplus g_{\S^{m-p-1}}  &  \text{ on } 
        H_2 =  \D^{p+1} \times {\Bbb S}^{m-p-1}.
 \end{cases}
\end{split}
\end{equation}
%%%%%%%%%%%%

 Since 
%%%%%%%%
\begin{equation*}
\begin{split}
 \vol(\S^p \times \D^{m-p}_L )
  &= \vol(\S^p) \cdot \vol( \D^{m-p}_L ) \\
  &= \vol(\S^p) \cdot \Big\{ \vol(\D^{m-p}) 
         + \vol(\S^{m-p-1}) \cdot L \Big\}, 
\end{split}
\end{equation*}
%%%%%%%
 we can write for some constants $A, B >0$ independent of $L$ 
%%%%%%%%
\begin{equation} \label{eq:vol-growth}
\begin{split}
   \vol (\S^m, g_{p,L}) &= \vol(\S^p \times \D^{m-p}_L ) 
        + \vol(\D^{p+1} \times \S^{m-p-1})  \\
    &= A L + B.
\end{split}
\end{equation}
%%%%%%%

 Next, we estimate the eigenvalues of the rough and Hodge Laplacians acting 
on $p$-forms from above.

%%%%%%%%%%%%%%%%%%%%%%%%%%%
\begin{lem} \label{lem:k-ev-estimate}
 For any integer $k \ge 1$ and any real number $L >0$, we have 
%%%%%%%%
\begin{enumerate}
 \item  $\ \overline{\lambda}^{(p)}_k ({\Bbb S}^m,g_{p,L})  
            \le  \dfrac{k^2 \pi^2}{L^2};$    
 \item  $\ \lambda^{\prime \prime (p)}_k ({\Bbb S}^m,g_{p,L}) 
            \le  \dfrac{k^2 \pi^2}{L^2}.$
\end{enumerate}
%%%%%%%
\end{lem}
%%%%%%%%%%%%%%%%%%%%%%%%%%%%%%%

%%%%%%%%
\begin{rem} \label{rem:no-zero-ev}
 We note that, for any metric, the rough and Hodge Laplacians acting 
 on $p$-forms of ${\Bbb S}^m$ for $1 \le p \le m-1$ have no $0$ eigenvalues.
 In fact, from $b_p(\S^m) = 0$ for $1 \le p \le m-1$, by the Hodge theory, 
 there exist no non-zero harmonic $p$-forms on ${\S}^m$.
 In particular, there exist no non-zero parallel $p$-forms.
\end{rem}
%%%%%%%%%

%%%%%%%%%%%%%%%%%
\begin{proof}
 We construct $k$ test $p$-forms $\varphi_i$ for the min-max principle. 
 Their behaviour will be like $f_iv_p$, for suitable functions $f_i$, if $v_p$ is
 the volume $p$-form on $(\Bbb{S}^{p}, g_{\Bbb{S}^{p}})$, so we can take advantage of
 the properties of the standard volume form. The functions $f_i$ are constructed as follows:
 we divide the interval $[2,L+2]$ of length $L$ into $k$ intervals 
 $I_i := [r_{i-1}, r_{i}]$ $(i =1,\dots,k)$, where
%%%%%%%%
\begin{equation*} 
\begin{split}
   2 = r_0 < r_1 <  \cdots < r_k = L+2  \  \text{ with } \
   r_i :=  \dfrac{L}{k} \, i + 2  \quad  (i= 0,1, \dots, k). 
\end{split}
\end{equation*}
%%%%%%%
 Let $f_i(r)$ be the first Dirichlet eigenfunction of the Laplacian 
 acting on functions on the interval $I_i$, that is, 
%%%%%%%%
\begin{equation*}
\begin{split}
  f_i(r) &= \sin \left( (r - r_{i-1}) \dfrac{k \pi}{L} \right)
   \quad  \text{ for } \  r \in  [r_{i-1}, r_{i}].
\end{split}
\end{equation*}
%%%%%%%

%%%%%%%%%%%%%%%%%%%%%%%%%%%%%%
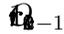
\begin{figure}[H] \label{fig:test-functions}
\begin{center}
\begin{tikzpicture} [scale=1.2]
 % \draw [help lines] (-1.0,-1.0) grid (7.5,2.5);
 % axis
  \draw[->] (-0.5,0) -- (6.5,0) node[below] {$r$};
  \draw[->] (0,-0.5) -- (0,1.5) node[left] {};
  \draw (0,0) node[below left] {$O$};  % the origin
%%%%
  \draw (1,0) node[below] {$r_0$}; % 1
  \draw (2,0) node[below] {$r_1$}; % 2
  \draw (3,0) node[below] {$r_2$}; % 3
  \draw (5,0) node[below] {$r_{k-1}$}; % L
  \draw (6,0) node[below] {$r_k$}; % L
  \draw (0,1) node[left] {$1$}; % 1
%%%%
  \draw [domain=0:1] plot (\x+1,{sin(\x*pi r)});
  \draw [domain=0:1] plot (\x+2,{sin(\x*pi r)});
  \draw [domain=0:1] plot (\x+3,{sin(\x*pi r)});
  \draw [domain=0:1] plot (\x+5,{sin(\x*pi r)});
  \draw [dotted] (0,1) -- (5.5,1);
%%%%%%%%%%%%
\end{tikzpicture}
   \vspace{-0.5cm}
\end{center}
   \caption{test functions $f_i(r)$}
\end{figure}
%%%%%%%%%%%%%%%%%%%%%%

Then, we define $k$ test $p$-forms $\vphi_i$ on ${\Bbb S}^m$ as follows
(recall that $v_p$ is the volume $p$-form on $(\Bbb{S}^{p}, g_{\Bbb{S}^{p}})$):
 on $H_1 = \S^{p} \times \D^{m-p}_{L}$, 
%%%%%%%%
\begin{equation} \label{eq:test-forms}
\begin{split}
 \vphi_i  &:=
 \begin{cases}
    f_i(r) v_p &  \text{on } 
        \S^{p} \times \Big( [r_{i-1}, r_{i}] \times \S^{m-p-1} \Big), \\
    \  0   &   \text{otherwise},
 \end{cases}
\end{split}
\end{equation}
%%%%%%%
and $\vphi_i \equiv 0$ on $H_2 = \D^{p+1} \times \S^{m-p-1}$.

 We remark that the family $\vphi_i$ $(i= 1,\dots,k)$ is orthogonal.

\noindent
$(1)$  We first prove the case of the rough Laplacian acting on $p$-forms.
 Since the orthogonal family $\vphi_i$ $(i= 1,\dots,k)$ have disjoint support, 
 they are also orthogonal for the quadratic form defining the rough Laplacian,
 namely, $\|\nabla \vphi\|^2_{L^2 (\S^m,g_{p,L})}$. 
 The min-max principle and Remark \ref{rem:no-zero-ev} give then
%%%%%%%%
\begin{equation} \label{eq:rough-min-max}
\begin{split}
 0 <  \overline{\lambda}^{(p)}_k ({\Bbb S}^m,g_{p,L})
   &\le  \max_{i=1,2,\dots,k} \left\{ \dfrac{ \| \nabla \vphi_i \|^2_{L^2 (\S^m,g_{p,L})}  
      }{ \| \vphi_i \|^2_{L^2(\S^m, g_{p,L})} } \right\}. 
\end{split}
\end{equation}
%%%%%%%
 Since $v_p$ is parallel, the numerator of the right-hand side of
 \eqref{eq:rough-min-max} is
%%%%%%%%
\begin{equation*}
\begin{split}
 \| \nabla \vphi_i \|^2_{L^2(\S^m, g_{p,L})}
   &=  \| \nabla (f_i v_p) \|^2_{L^2 (\S^{p} \times \D^{m-p}_{L})} 
    =  \| df_i \otimes v_p \|^2_{L^2 (\S^{p} \times (I_i \times {\Bbb S}^{m-p-1}) )}. 
\end{split}
\end{equation*}
%%%%%%%
 Since the Riemannian metric on  $C_i = I_i \times {\Bbb S}^{m-p-1}$ is product, we have
%%%%%%%%
\begin{equation*}
\begin{split}
 \| df_i \otimes v_p \|^2_{L^2 (\S^{p} \times C_i)} 
  &= \dint_{\S^p} \dint_{I_i \times {\Bbb S}^{m-p-1}} | df_i \otimes v_p |^2 
        d \mu_{\S^p} d \mu_{C_i} \\
  &= \vol(\S^p) \dint_{I_i \times \S^{m-p-1}} | df_i |^2 dr d \mu_{\S^{m-p-1}} \\
  &= \vol(\S^p) \vol( \S^{m-p-1}) \dint^{r_i}_{r_{i-1}} | df_i |^2 dr.
\end{split}
\end{equation*}
%%%%%%%
 Since $f_i(r)$ is the Dirichlet eigenfunction on the interval $I_i = [r_{i-1},r_i]$,
 which is isometric to $[0, \frac{L}{k}]$, we have 
%%%%%%%%
\begin{equation*} \label{eq:eigenfunc}
\begin{split}
  \dint^{r_i}_{r_{i-1}} | df_i |^2 dr 
     &= \dfrac{k^2 \pi^2}{L^2} \, \dint^{r_i}_{r_{i-1}} | f_i |^2 dr.
\end{split}
\end{equation*}
%%%%%%%
 Therefore, the numerator of the right-hand side of \eqref{eq:rough-min-max} is
%%%%%%%%
\begin{equation} \label{eq:numerator-2}
\begin{split}
 \| \nabla \vphi_i \|^2_{L^2(\S^m, g_{p,L})}
   &=  \dfrac{k^2 \pi^2}{L^2} \, \vol(\S^p) \vol( \S^{m-p-1}) \, 
        \dint^{r_i}_{r_{i-1}} | f_i |^2 dr.
\end{split}
\end{equation}
%%%%%%%

 On the other hand, the denominator of the right-hand side of  
 \eqref{eq:rough-min-max} is
%%%%%%%%
\begin{equation} \label{eq:denomirator-1}
\begin{split}
 \| \vphi_i \|^2_{L^2(S^m, g_{p,L})}
   &=  \| f_i v_p \|^2_{L^2 (\S^{p} \times \D^{m-p}_{L})} 
    =  \| f_i v_p \|^2_{L^2 (\S^{p} \times C_i)}  \\
   &= \vol(\S^p) \dint_{I_i \times \S^{m-p-1}} | f_i |^2 dr d \mu_{\S^{m-p-1}} \\
  &= \vol(\S^p) \vol( \S^{m-p-1}) \dint^{r_i}_{r_{i-1}} | f_i |^2 dr.
\end{split}
\end{equation}
%%%%%%%
 Thus, by substituting \eqref{eq:numerator-2} and \eqref{eq:denomirator-1}
 into \eqref{eq:rough-min-max}, we obtain 
%%%%%%%%
\begin{equation*}
\begin{split}
  0 < \overline{\lambda}^{(p)}_k ({\Bbb S}^m,g_{p,L})
    &\le  \dfrac{k^2 \pi^2}{L^2}.
\end{split}
\end{equation*}
%%%%%%%%%%%%
\noindent
 $(2)$  Next, we prove the case of the Hodge Laplacian acting on co-exact $p$-forms. 
 We use the same test $p$-forms $\vphi_i$ constructed in \eqref{eq:test-forms}.
 For the same reason as before, this family is orthogonal for the quadratic form
 defining the Hodge Laplacian $\|(d + \delta) \vphi\|^2_{L^2(\S^m, g_{p,L})}$.
 Moreover, we note that the test $p$-forms $\vphi_i$ are co-closed.
 Indeed, since the Riemannian metric is product on the support of $\vphi_i$, 
%%%%%%%%
\begin{equation*}
\begin{split}
 \delta_{g_{p,L}} \vphi_i 
   &= f (\delta_{g_{p,L}} v_p) - i_{({\grad}_{g_{p,L}} f)}(v_{p}) \equiv 0.
\end{split}
\end{equation*}
%%%%%%%%%%%%
 Since ${\Bbb S}^m$ has no non-zero harmonic $p$-forms, all co-closed forms must 
 be co-exact. Therefore, from the min-max principle of the Hodge-Laplacian acting 
 on co-exact $p$-forms, it follows that 
%%%%%%%%
\begin{equation*} \label{eq:Hodge-min-max}
\begin{split}
  0 < \lambda^{\prime \prime (p)}_k ({\Bbb S}^m, g_{p,L})
   &\le  \max_{i=1,2,\dots,k} \left\{  \dfrac{ \| d \vphi_i \|^2_{L^2(\S^m,g_{p,L})} 
      }{ \| \vphi_i \|^2_{L^2(\S^m, g_{p,L})} } \right\}. 
\end{split}
\end{equation*}
%%%%%%%
 By the same calculations as in $(1)$, we obtain the same upper bound 
%%%%%%%%
\begin{equation*}
\begin{split}
  \lambda^{\prime \prime (p)}_k ({\Bbb S}^m,g_{p, L})
    &\le  \dfrac{k^2 \pi^2}{L^2}.
\end{split}
\end{equation*}
%%%%%%%
\end{proof}
%%%%%%%%%%%%%%%%%%%%%%%%%%%%%%%%%%

 Finally, we normalize the volume to be one.
 Namely, if we set a new Riemannian metric 
%%%%%%%%
\begin{equation*}
\begin{split}
  \overline{g}_{p,L} &:= \vol({\Bbb S}^m, g_{p,L})^{-\frac{2}{m}} g_{p,L}, 
\end{split}
\end{equation*}
%%%%%%%
 then 
%%%%%%%%
%\begin{equation*}
%\begin{split}
 $ \vol({\Bbb S}^m, \overline{g}_{p,L}) \equiv 1$
%\end{split}
%\end{equation*}
%%%%%%%
 and still $K_{\overline{g}_{p,L}} \ge 0.$
 From Lemma \ref{lem:k-ev-estimate} and \eqref{eq:vol-growth}, we have
%%%%%%%%
\begin{equation} \label{eq:ev-vol}
\begin{split}
 \overline{\lambda}^{(p)}_k ({\Bbb S}^m, \overline{g}_{p,L})
  &=  \overline{\lambda}^{(p)}_k ({\Bbb S}^m, g_{p,L}) \cdot 
          \vol({\Bbb S}^m, g_{p,L})^{\frac{2}{m}} \\
   &\le  \dfrac{k^2 \pi^2}{L^2} \cdot ( A L + B )^{\frac{2}{m}} \\
   &=  k^2 \pi^2 \cdot \left( \dfrac{A L + B}{L^m} \right)^{\frac{2}{m}}
   \longrightarrow 0,
\end{split}
\end{equation}
%%%%%%%
 and similarly  
 $\lambda^{\prime \prime (p)}_k ({\Bbb S}^m, \overline{g}_{p,L}) \longrightarrow 0$
 as $L \longrightarrow \infty.$
 Thus, we have finished the proof of Theorem \ref{thm:sphere}.
\end{proof}
%%%%%%%%%%%%%%%%%%%%%%

%%%%%%%%%%%%%%%
\begin{rem}\label{conj:lott}
 Theorem \ref{thm:sphere} implies that the first positive eigenvalue of 
 the Hodge-Laplacian acting on $p$-forms cannot be estimated below 
 in terms of dimension, volume and a lower bound of the sectional curvature. 
 For this family $\overline{g}_{p,L}$, 
 the diameter $\diam ({\Bbb S}^m,\overline{g}_{p,L}) \longrightarrow \infty$
 as $L \to \infty.$

 Lott \cite[p.918]{Lott[04]-quotient} conjectured that 
 for given $m \in \N$, $\kappa \in \R$ and $v, D >0$, 
 there exists a positive constant $C(m, \kappa,v,D) >0$ such that any connected
 oriented closed Riemannian manifold $(M^m,g)$ of dimension $m$ with the sectional 
 curvature $K_g \ge \kappa$, the volume $\vol(M,g) \ge v$ and the diameter
 $\diam (M,g) \le D$ satisfies 
%%%%%%%%
\begin{equation*} 
\begin{split}
  \lambda^{(p)}_1(M,g) &\ge C(m, \kappa,v,D) > 0. 
\end{split}
\end{equation*}
%%%%%%%
 This conjecture is still open.
 We note that this is a non-collapsing case.
 In a collapsing case, we do not know anything, but recently Boulanger and Courtois
 \cite{Boulanger-Courtois[21]} proposed a Cheeger constant for coexact 1-forms, 
 whose square gives a lower bound of the first positive eigenvalue of the Hodge-Laplacian 
 acting on $1$-forms for Riemannian manifolds with bounded diameter and 
 bounded sectional curvature.
\end{rem}
%%%%%%%%%%%%%%%

%%%%%%%%%%%%%%%%%%%%%%%%%%%%%%%%%%%%%%%%%%%%%%%%%%%%%%%%%%%%%%%%%%%%%%%%%
%%%%%%%%               SECTION      
%%%%%%%%%%%%%%%%%%%%%%%%%%%%%%%%%%%%%%%%%%%%%%%%%%%%%%%%%%%%%%%%%%%%%%%%%

\section{Lower bounds for the eigenvalues of the Hodge-Laplacian on ${\Bbb S}^m$}

 We consider lower bounds of the eigenvalues of the Hodge-Laplacian acting 
 on exact $q$-forms for $1 \le q \le m$ for the one-parameter family of 
 volume un-normalized Riemannian metrics $g_{p,L}$ on ${\Bbb S}^m$ constructed 
 in Theorem \ref{thm:sphere}.

%%%%%%%%%%%%%%%%%%% thm:lower-bdd %%%%%%%%%%%%%%%%%%%%%%%%%%%%%%%%%%
\begin{thm} \label{thm:lower-bdd}
 Let $p$ be an integer with $1 \le p \le m-1$.
 For the one parameter family of Riemannian metrics $g_{p,L}$ on ${\Bbb S}^m$ 
 constructed in the proof of Theorem \ref{thm:sphere}, 
 the eigenvalues of the Hodge-Laplacian acting on exact $q$-forms 
 for $1 \le q \le m$ satisfy the following: 
%%%%%%%%%%%%
\begin{enumerate}
 %% display (i), (ii)
 % \renewcommand{\labelenumi}{$(\roman{enumi})$}
%%%%
  \item  For $q \neq 1,p,p+1,m-p-1,m-p,m-1,m,$
 there exists a positive constant $C > 0$ independent of $L$ such that
%%%%%%%%%%%
\begin{equation*}
\begin{split}
 \lambda^{\prime (q)}_{1}({\Bbb S}^m, g_{p,L}) &\ge  C >0.
\end{split}
\end{equation*}
%%%%%%%%%
  \item  For $q = 1,p,p+1,m-p-1,m-p,m-1,m,$ 
 there exist positive constants $C_1, C_2 > 0$ independent of $L$ 
 such that for sufficiently large $L >0$ 
%%%%%%%%%%%
\begin{equation*}
\begin{split}
 \lambda^{\prime (q)}_{n_q +1}({\Bbb S}^m, g_{p,L}) &\ge \dfrac{1}{C_1 L^2 + C_2}.
\end{split}
\end{equation*}
%%%%%%%%%
 Here, $n_q$ is given by 
%%%%%%%%
\begin{equation*}
\begin{split}
  n_q &:= 
\begin{cases}
   4  &  \text{if } (p,q) = (\dfrac{m-1}{2}, \dfrac{m+1}{2})
              \ \text{ and $m$ is odd,} \\
   2  &  \text{if }  q = 1, p+1,m-p,m,  \\
      &  \quad \text{ except for } (p,q) = (\dfrac{m-1}{2}, \dfrac{m+1}{2})
          \text{ if $m$ is odd,} \\
   0  &  \text{otherwise}.
\end{cases}
\end{split}
\end{equation*}
%%%%%%%
\end{enumerate}
\end{thm}
%%%%%%%%%%%%%%%%%%%%%%%%%%%%%%%%%%%%%%%

%%%%%%%%%%%%%%%%%%% cor:lower-bdd %%%%%%%%%%%%%%%%%%%%%%%%%%%%%%%%%%
\begin{cor} \label{cor:lower-bdd}
 Let $p$ be an integer with $1 \le p \le m-1$.
 For the one parameter family of the volume normalized Riemannian metrics
 $\overline{g}_{p,L}$ on ${\Bbb S}^m$, 
 if $q \neq 1,p,p+1,m-p-1,m-p,m-1,m,$ we have
%%%%%%%%%%%
\begin{equation*}
\begin{split}
 \lambda^{(q)}_{1}({\Bbb S}^m, \overline{g}_{p,L}) 
  &\longrightarrow \infty  \quad
    \text{as } \  L \to \infty.
\end{split}
\end{equation*}
\end{cor}
%%%%%%%%%%%%%%%%%%%%%%%%%%%%%%%%%%%%%%%

%%%%%%%%
\begin{proof}[Proof of Corollary $\ref{cor:lower-bdd}$]
 By combining $(1)$ in Theorem \ref{thm:lower-bdd} with the Hodge-duality
  $\lambda^{\prime \prime (q)}_{1} = \lambda^{\prime (m-q)}_{1}$,
 we find that the first positive eigenvalue of the Hodge-Laplacian acting 
 on $q$-forms for $q \neq 1,p,p+1,m-p-1,m-p,m-1$ have a uniform lower bound 
 in $L$.
 That is, there exists a positive constant $C >0$ independent of $L$ such that 
%%%%%%%%%%%
\begin{equation*}
\begin{split}
 \lambda^{(q)}_{1}({\Bbb S}^m, g_{p,L}) 
  &= \min \{ \lambda^{\prime (p)}_1({\Bbb S}^m, g_{p,L}, 
         \ \lambda^{\prime \prime (p)}_1({\Bbb S}^m, g_{p,L}) \} \\
  &= \min \{ \lambda^{\prime (p)}_1({\Bbb S}^m, g_{p,L},
         \ \lambda^{\prime (m-p)}_1({\Bbb S}^m, g_{p,L}) \} \\
  &\ge  C >0.
\end{split}
\end{equation*}
%%%%%%%%%
 For the volume normalized metric 
%%%%%%%%
%\begin{equation*}
%\begin{split}
 $ \overline{g}_{p,L} = \vol({\Bbb S}^m, g_{p,L})^{-\frac{2}{m}} g_{p,L}$, 
%\end{split}
%\end{equation*}
%%%%%%%
 in the same way as in \eqref{eq:ev-vol}, we have
%%%%%%%%
\begin{equation*} %\label{eq:ev-vol}
\begin{split}
 \lambda^{(q)}_1 ({\Bbb S}^m, \overline{g}_{p,L})
   &= \lambda^{(q)}_1 ({\Bbb S}^m, g_{p,L}) \cdot 
          \vol({\Bbb S}^m, g_{p,L})^{\frac{2}{m}} \\
   &\ge  C (A L + B)^{\frac{2}{m}} 
   \longrightarrow \infty,  \quad  \text{ as } \  L \to \infty.
\end{split}
\end{equation*}
%%%%%%%
\end{proof}
%%%%%%%%%%%%%%%%%

%%%%%%%%
\begin{rem}
 In the case of $n_q =2, 4$ for $q = 1,p,p+1,m-p-1,m-p,m-1,m,$ 
 we do not know whether or not 
 $\lambda^{(q)}_{i}({\Bbb S}^m, g_{p,L})$ for $i=1, \dots, n_q$ 
 have positive lower bounds independent of $L$. 
 A similar problem occurs in \cite[Remark $5.7$]{Egidi-Post[17]}, p.457.
\end{rem}
%%%%%%%%%%%%%%%%%

 We prove Theorem \ref{thm:lower-bdd} in the same way as Gentile and Pagliara 
 \cite{Gentile-Pagliara[95]}.
 In this way, the following result by McGowan \cite[Lemma $2.3$]{McGowan[93]} 
 (see also \cite{Gentile-Pagliara[95]}, Lemma $1$) plays an important r\^ole.
 A similar argument was used to prove Corollary $1.5$ in \cite{large-vn[05]}.

\vspace{0.3cm}
 We now denote by $\nu^{\prime (p)}_1(U,g)$ the first positive eigenvalue of 
 the Hodge-Laplacian acting on exact $p$-forms on $(U,g)$ 
 with the absolute boundary condition.

%%%%%%%%%%%%%%%%%   lem: McGowan  %%%%%%%%%%%%%%%%%%%%%%%%%%%%
\begin{lem}[McGowan \cite{McGowan[93]}] \label{lem:McGowan} 
 Let $(M^m,g)$ be a connected oriented closed Riemannian manifold of dimension $m$.
 We take a finite open covering $\{ U_i \}_{i=1}^K $ of $M$ satisfying
   $U_i \cap U_j \cap U_k = \emptyset$ and
 a partition of unity $ \{ \rho_i \}_{i=1}^K$ subordinated to $\{ U_i \}_{i=1}^K $.
 If we set $n_p := \displaystyle \sum_{i <j} \dim H^{p-1}(U_{ij}; \R)$,
  where $U_{ij} := U_i \cap U_j,$ and 
%%%%%%%%
\begin{equation*}
\begin{split}
 C_g(\rho) &:= \max_{i=1,\ldots,K} \max_{x \in U_i} \big\{ | d \rho_i|^2_g(x) \big\},
\end{split}
\end{equation*}
%%%%%%%    
 then we have 
%%%%%%%%%%%%%%%%%%%%
\begin{equation*}
\begin{split}
 \lambda^{\prime (p)}_{n_p+1} & (M,g) \ge  \\
   & \dfrac{1}{8 \dsum_{i=1}^K \left\{ \dfrac{1}{ \nu^{\prime (p)}_1(U_i,g) } 
     + \dsum_{\substack{ j \neq i \\ U_i \cap U_j \neq \emptyset}}
        \left( \dfrac{ C_g(\rho)}{\nu^{\prime (p-1)}_1(U_{ij},g)}
       + 1 \right) \left( \dfrac{1}{ \nu^{\prime (p)}_1(U_i,g) } + 
       \dfrac{1}{ \nu^{\prime (p)}_1(U_j,g) } \right)  \right\} }.
\end{split}
\end{equation*}
%%%%%%%%%%%%%%%%%%%%%
\end{lem}
%%%%%%%%%%%%%%%%%%%%%%%%%%%%%%%%%%%%%%%%%%%%%%%%%%%%%%%%%%%%%%%%%%%%%%

%%%%%%%%%%%%%%%%%%%%%%%%%
\begin{proof}[Proof of Theorem $\ref{thm:lower-bdd}$.]
 We take an open covering $\{ U_1, U_2, U_3\}$ of $M$ as follows:
%%%%%%%%
\begin{equation*}
\begin{split}
   U_1 &:= {\S}^p \times \big( [0,3] \times {\S}^{m-p-1} \big), \\
   U_2 &:= {\S}^p \times \big( [2,L+2] \times {\S}^{m-p-1} \big), \\
   U_3 &:= \Big( {\S}^p \times \big( [L+1,L+2] \times {\S}^{m-p-1} 
      \big) \Big) \cup \Big( \D^{p+1} \times \S^{m-p-1} \Big).
\end{split}
\end{equation*}
%%%%%%%
 Then, 
%%%%%%%%
\begin{equation*}
\begin{split}
  U_{12} &= U_1 \cap U_2 = {\S}^p \times \big( [2,3] \times {\S}^{m-p-1} \big), \\
  U_{13} &= U_1 \cap U_3 = \emptyset,  \\
  U_{23} &= U_2 \cap U_3 = {\S}^p \times \big( [L+1,L+2] \times {\S}^{m-p-1} \big)
\end{split}
\end{equation*}
%%%%%%%
 and $U_{123} := U_1 \cap U_2 \cap U_3 = \emptyset.$
 Since both $U_{12}$ and $U_{23}$ are isometric to 
 ${\S}^p \times \big( [0,1] \times {\S}^{m-p-1} \big),$
  their eigenvalues do not depend on $L$.
 Therefore, only $\nu^{\prime (p)}_1(U_2, g_{p,L})$ depends on $L$.

 We can take a partition of unity $\{ \rho_i \}_{i=1,2,3}$ subordinate to 
 this open covering $\{ U_1, U_2, U_3 \}$ such that
 the supports of $d \rho_i$ are in $U_{12}$ or $U_{23}$.
 Since the $C^0$-norms of $d \rho_i$ are independent of $L$,
  the constant $C_g(\rho)$ is also independent of $L$.

 \vspace{0.3cm}
%%%%%%%%%%%

 Now, we compute $n_q.$
 Since both $U_{12}$ and $U_{23}$ are homotopy equivalent to ${\S}^p \times {\S}^{m-p-1},$
 by the K\"unneth formula, we have 
%%%%%%%%
\begin{equation*}
\begin{split}
  n_q &= \dim H^{q-1}(U_{12}; \R) + \dim H^{q-1}(U_{23}; \R)  \\
      &= 2 \dim H^{q-1}({\S}^p \times {\S}^{m-p-1}; \R) \\
      &= 
\begin{cases}
   4  &  \text{if } (p,q) = (\dfrac{m-1}{2}, \dfrac{m+1}{2})
              \ \text{ and $m$ is odd,} \\
   2  &  \text{if }  q = 1, p+1,m-p,m,  \\
      &  \quad \text{ except for } (p,q) = (\dfrac{m-1}{2}, \dfrac{m+1}{2})
          \text{ if $m$ is odd,} \\
   0  &  \text{otherwise}.
\end{cases}
\end{split}
\end{equation*}
%%%%%%%

 Next, we estimate the eigenvalue $\nu^{\prime (q)}_1(U_2,g_{p,L})$ from below.
 If $0$ is the eigenvalue for $s$-forms on $(N,h)$, 
 we denote it by $\lambda^{(s)}_0 (N,h).$
 From the K\"unneth formula for the eigenvalues of the Hodge-Laplacian
 (e.g., \cite{Gilkey-Leahy-Park[99]}, p.38, Example $1.5.7$), we have
%%%%%%%%%%%%%%%%%%%%
\begin{equation} \label{eq:K"unneth-formula}
\begin{split}
 \nu^{\prime (q)}_1(U_2,g_{p,L})
  &\ge  \nu^{(q)}_1(U_2,g_{p,L})
   =  \nu^{(q)}_1([0,L] \times {\S}^p \times {\S}^{m-p-1})  \\
  &=  \min_{\substack{a+b=q \\ i+j \ge 1}} \Big\{ 
        \nu^{(a)}_i([0,L]) + \lambda^{(b)}_j({\S}^p \times {\S}^{m-p-1}) \Big\} \\
  &=  \min_{\substack{a+b=q \\ i+j \ge 1}} \Big\{ 
        L^{-2} \, \nu^{(a)}_i([0,1]) + \lambda^{(b)}_j({\S}^p \times {\S}^{m-p-1}) \Big\}. 
\end{split}
\end{equation}
%%%%%%%%%%%%
 To proceed with the calculation, we consider whether or not $0$ is the eigenvalue
 on $[0,1]$ and ${\S}^p \times {\S}^{m-p-1}$.
 By the Hodge theory, this follows from the cohomology groups 
%%%%%%%%
\begin{equation} \label{eq:cohomology}
\begin{split}
 H^a ([0,1];\R) &
\begin{cases}
  \neq  0   &   \text{if } a=0, \\
  =     0   &   \text{if } a=1,
\end{cases}  \\[0.2cm] 
  H^{b}({\S}^p \times {\S}^{m-p-1}; \R) &
\begin{cases}
  \neq  0   &   \text{if } b = 0, p, m-p-1, m-1, \\
  =     0   &   \text{otherwise}.
\end{cases}
\end{split}
\end{equation}
%%%%%%%

 If $q = p, m-p-1, m-1$, by \eqref{eq:cohomology}, 
 $0$ is the eigenvalue for $q$-forms on ${\S}^p \times {\S}^{m-p-1}$.
 Hence, we have for sufficiently large $L,$ 
%%%%%%%%
\begin{equation*}
\begin{split}
 \nu^{\prime (q)}_1 (U_2,g_{p,L}) 
  &\ge  \min_{\substack{a+b=q \\ i+j \ge 1}} \Big\{ 
        L^{-2} \, \nu^{(a)}_i([0,1])
        + \lambda^{(b)}_j({\S}^p \times {\S}^{m-p-1}) \Big\}  \\ 
  &= \min_{j \ge 0} \Big\{ \, L^{-2} \underbrace{ \nu^{(0)}_0([0,1]) }_{=0} 
        + \lambda^{(q)}_1({\S}^p \times {\S}^{m-p-1}), \
       L^{-2} \nu^{(0)}_1([0,1]) + 
        \underbrace{ \lambda^{(q)}_0({\S}^p \times {\S}^{m-p-1}) }_{=0}, \\
  &\qquad \qquad  L^{-2} \nu^{(1)}_1([0,1]) 
        + \lambda^{(q-1)}_j({\S}^p \times {\S}^{m-p-1}) \, \Big\} \\
  &\ge  L^{-2} \min \Big\{ \, \nu^{(0)}_1([0,1]), \
           \nu^{(1)}_1([0,1]) \, \Big\}.
\end{split}
\end{equation*}
%%%%%%%
 Similarly, if $q = 1, p+1, m-p, m$, by \eqref{eq:cohomology}, 
 $0$ is the eigenvalue for $(q-1)$-forms on ${\S}^p \times {\S}^{m-p-1}$.
 Hence, we have for sufficiently large $L,$ 
%%%%%%%%
\begin{equation*}
\begin{split}
 \nu^{\prime (q)}_1 (U_2,g_{p,L}) 
  &\ge  \min_{\substack{a+b=q \\ i+j \ge 1}} \Big\{ 
        L^{-2} \, \nu^{(a)}_i([0,1])
        + \lambda^{(b)}_j({\S}^p \times {\S}^{m-p-1}) \Big\}  \\ 
  &= \min_{j \ge 0} \Big\{ \, L^{-2} \underbrace{ \nu^{(0)}_0([0,1]) }_{=0} 
        + \lambda^{(q)}_1({\S}^p \times {\S}^{m-p-1}), \
       L^{-2} \nu^{(0)}_1([0,1]) + \lambda^{(q)}_j({\S}^p \times {\S}^{m-p-1}), \\
  &\qquad \qquad   L^{-2} \nu^{(1)}_1([0,1]) + 
        \underbrace{ \lambda^{(q-1)}_0({\S}^p \times {\S}^{m-p-1})}_{=0} \, \Big\} \\
  &\ge  L^{-2} \min \Big\{ \, \nu^{(0)}_1([0,1]), \
           \nu^{(1)}_1([0,1]) \, \Big\}.
\end{split}
\end{equation*}
%%%%%%%
 Therefore, if $q = 1, p, p+1, m-p-1, m-p, m-1, m$, from Lemma \ref{lem:McGowan},
 there exist positive constants $C_1, C_2 > 0$
 independent of $L$ such that 
%%%%%%%%
\begin{equation*}
\begin{split}
  \lambda^{\prime (q)}_{n_q +1} ({\Bbb S}^{m}, g_{p,L})
      &\ge \dfrac{1}{C_1 L^2 + C_2}.
\end{split}
\end{equation*}
%%%%%%%

 If $q \neq  1, p, p+1, m-p-1, m-p, m-1, m$, by \eqref{eq:cohomology}, 
 $0$ is neither eigenvalue for $(q-1)$-forms nor for $q$-forms on
 ${\S}^p \times {\S}^{m-p-1}$.
 Hence, for any $L >0$, we have 
%%%%%%%%
\begin{equation*}
\begin{split}
  \nu^{\prime (q)}_1(U_2,g_{p,L}) 
  &\ge \min \Big\{ \, \lambda^{(q)}_1({\S}^p \times {\S}^{m-p-1}), \
      L^{-2} \nu^{(0)}_1([0,1]) + \lambda^{(q)}_1({\S}^p \times {\S}^{m-p-1}), \\
  &\qquad \qquad  L^{-2} \nu^{(1)}_1([0,1]) + 
      \lambda^{(q-1)}_1({\S}^p \times {\S}^{m-p-1}) \Big\} \\
  &\ge \min \Big\{ \, \lambda^{(q)}_1({\S}^p \times {\S}^{m-p-1}), \
          \lambda^{(q-1)}_1({\S}^p \times {\S}^{m-p-1}) \Big\}  >0.
\end{split}
\end{equation*}
%%%%%%%
 In this case, we have $n_q = 0$.
 Thus, we obtain a lower bound $C >0$ independent of $L$ such that 
%%%%%%%%
\begin{equation*}
\begin{split}
   \lambda^{\prime (q)}_1 ({\Bbb S}^{m}, g_{p,L}) \ge C >0.
\end{split}
\end{equation*}
%%%%%%%
\end{proof}
%%%%%%%%%%%%%%%%%%%%%%%%%%%%

%%%%%%%%%%%%%%%%%%%%%%%%%%%%%%%%%%%%%%%%%%%%%%%%%%%%%%%%%%%%%%%%%%%%%%%%%
%%%%%%%%               SECTION      
%%%%%%%%%%%%%%%%%%%%%%%%%%%%%%%%%%%%%%%%%%%%%%%%%%%%%%%%%%%%%%%%%%%%%%%%%

\section{General manifold}

\subsection{Gluing theorem}

 To prove Theorem \ref{thm:main-thm}, we need a gluing theorem for the eigenvalues
 on a connected sum.
 The gluing theorem we use here is obtained from the convergence of the eigenvalues 
 of the Laplacian, when one side of a connected sum of two closed Riemannian manifolds
 collapses to a point. We call it collapsing of connected sums.
 This was studied in the case of the Laplacian acting on functions in \cite{0-conv[02]},
 and in the case of the Hodge-Laplacian acting on $p$-forms in \cite{p-gap[03]}, 
 \cite{p-conv[12]}. We recall it.

 Let $(M_i,g_i)$, $i=1,2$, be connected oriented closed Riemannian manifolds of 
 the same dimension $m$ ($m \ge 2$). For simplicity, we suppose that each metric 
 $g_i$ is flat on the geodesic ball $B(x_i,r_i)$ with the radius $r_i>0$ centered 
 at $x_i \in M_i$, where $r_i$ is smaller than the injectivity radius of $(M_i,g_i)$. 
 By changing the scale of $g_2$, we may suppose $r_2 = 2$. 
 Set $M_i(r) := M_i \setminus B(x_i,r)$. 
 For any $\e > 0$ with $0< \e < \min \{ r_1, 1 \},$
 since both boundaries of $\partial (M_1(\e), g_{1})$ and 
 $\partial (M_2(1), {\e}^2 g_{2})$ are isometric to the $(m-1)$-dimensional sphere
 of radius $\e$ in ${\R}^m$,
 we glue $(M_1(\e),g_1)$ to $(M_2(1), {\e}^2 g_2)$ along their boundaries.
 After deforming $g_2$ on a neighborhood of $\partial M_2(1)$,
 we obtain the new closed {\bf smooth} Riemannian manifold 
%%%%%%%%
\begin{equation} \label{eq:conn-sums}
\begin{split}
 (M, g_{\e}) &:= (M_1(\e),g_1) \cup_{\partial} (M_2(1), \e^2 g_2).
\end{split}
\end{equation}
%%%%%%%
 If we choose suitable orientations of $M_1$ and $M_2$, 
 $M$ is also oriented.

 From the construction of $(M, g_{\e})$, it is easy to see that 
%%%%%%%%
\begin{equation} \label{eq:vol-conv}
\begin{split}
  \dlim_{\e \to 0} \vol(M, g_{\e}) &= \vol(M_1,g_1).
\end{split}
\end{equation}
%%%%%%%

 In our previous works \cite{p-conv[12]}, \cite{0-conv[02]}, we have the following 
 convergence theorem for the eigenvalues of the Hodge-Laplacian acting 
 on exact and co-exact $p$-forms.
 In fact, by considering the convergence of eigenforms, we find that 
 all the eigenvalues for exact and co-exact forms still converge. 

%%%%%%%%%%%%%   lemma:p-conv   %%%%%%%%%%%%%%%%%%%%%%%%
\begin{lem}\label{lem:p-conv}
 For all $k=1,2, \dots,$ we have 
%%%%%%%%
\begin{equation*}
\begin{split}
  \lim_{\e \rightarrow 0} \lambda^{\prime (p)}_k(M,g_{\e})
     &= \lambda^{\prime (p)}_{k} (M_1,g_1),  \\
  \lim_{\e \rightarrow 0} \lambda^{\prime \prime (p)}_k(M,g_{\e})
     &= \lambda^{\prime \prime (p)}_{k} (M_1,g_1).
\end{split}
\end{equation*}
%%%%%%%%%
\end{lem}
%%%%%%%%%%%%%%%%%%%%%%

 We also have the convergence of the eigenvalues of the rough Laplacian 
 acting on $p$-forms.

%%%%%%%%%%%%%   thm : rough-conv  %%%%%%%%%%%%%%%%%%%%%%%%
\begin{thm}
 For all $k=1,2, \dots,$ we have
%%%%%%%%
\begin{equation*}
\begin{split}
  \lim_{\e \rightarrow 0} \overline{\lambda}^{(p)}_k(M,g_{\e}) &= 
  \overline{\lambda}^{(p)}_{k} (M_1,g_1).
\end{split}
\end{equation*}
\end{thm}
%%%%%%%%%%%%%%%%%%%%%%

 In fact, in the same way as the proof of Theorem $4.4$ in \cite{p-gap[03]}, p.$21$,
 we see the upper bound for the eigenvalues of the rough Laplacian 
 acting on $p$-forms.

%%%%%%%%%%%%%   lemma : upper bound   %%%%%%%%%%%%%%%%%%%%%%%%
\begin{lem} \label{lem:upper-bound}
 For all $k=1,2, \dots,$ we have
%%%%%%%%
\begin{equation*}
\begin{split}
  \limsup_{\e \rightarrow 0} \overline{\lambda}^{(p)}_k(M,g_{\e}) &\le 
  \overline{\lambda}^{(p)}_{k} (M_1,g_1).
\end{split}
\end{equation*}
\end{lem}
%%%%%%%%%%%%%%%%%%%%%%

 On the other hand, we will give the proof of the lower bound for the eigenvalues 
 of the rough Laplacian acting on $p$-forms in Section $6$, Appendix.

%%%%%
\begin{proof}[Proof of Lemma $\ref{lem:upper-bound}$]
 To prove Lemma \ref{lem:upper-bound}, we use a standard cut-off argument 
 for the min-max principle for eigenvalues of the rough Laplacian.
 Let $\{ \varphi_1, \dots, \varphi_k \}$ be an $L^2(M_1,g_1)$-orthonormal system
 of the eigen $p$-forms of the rough Laplacian on $(M_1,g_1)$ associated with 
 the eigenvalue $\overline{\lambda}^{(p)}_j(M_1,g_1)$ for $j = 1,2, \dots, k$.
 We take a cut-off function $\chi_{\e}(r)$ on $M_1$ defined as
%%%%%%%%%%%%%%%%%%%%
\begin{eqnarray}  \label{eq:cut-off-func}
  \chi_{\varepsilon}(r) :=
\begin{cases}
   \qquad  0   &   \   ( 0 \le r \le \e ),  \\[0.2cm]
   - \dfrac{2}{\log \e}  \log\left( \dfrac{r}{\e} \right) 
               &   \   ( \e \le r \le \sqrt{\e}),  \\[0.2cm]
   \qquad  1   &   \   ( \sqrt{\e} \le r),
\end{cases}
\end{eqnarray}
%%%%%%%%%%%%%%%%%%
 where $r$ is the Riemannian distance from $x_1 \in M_1$ with respect to $g_1.$ 
 We take a linear subspace $E_{\e}$ in $H^1 (\Lambda^p M, g_1)$ spanned by
 $\{ \chi_{\e} \varphi_1, \dots, \chi_{\e} \varphi_k  \},$ 
 and we see $\dim E_{\e} = k$.
 If we take this subspace $E_{\e}$ as a test $k$-dimensional subspace
 for the min-max principle for the eigenvalues of the rough Laplacian
 acting on $p$-forms, we obtain 
%%%%%%%%%%%%%%%%%%
\begin{equation} \label{eq:upper-bdd}
\begin{split}
 \overline{\lambda}^{(p)}_{k}(M,g_{\e})
  &\le  \sup_{\vphi_{\e} \neq 0 \in E_{\e}} 
        \left\{ \dfrac{ \| \nabla \vphi_{\e} \|^2_{L^2(M,g_{\e})} }
              { \quad \| \vphi_{\e} \|^2_{L^2(M,g_{\e})} } \right\}  
   \le   \overline{\lambda}^{(p)}_{k}(M_1,g_1) + \delta(\e),
\end{split}
\end{equation} 
%%%%%%%%%%%%%%%%%%
 where $\delta(\e) \to 0$ as $\e \to 0$.
 For details, see the proof of Theorem $4.4$ in \cite{p-gap[03]}, p.$21$.
\end{proof}

%%%%%%%%%%%%%%%%%%%%%%%%%%%%%%%%%%%%%%%%%%%%%%%%%%%%%%%%%%%%%%%%%%

\subsection{Proof of Theorem \ref{thm:main-thm}}

 We prove Theorem \ref{thm:main-thm} for a general manifold $M$.
 The main idea is to perform a connected sum of this manifold $M$ and 
 the sphere ${\Bbb S}^m$ equipped with the Riemannian metric constructed in 
 Theorem \ref{thm:sphere}.

%%%%%%%%%%%%%%%%%
\begin{proof}[{\it Proof of Theorem $\ref{thm:main-thm}$}]
 Let $M^m$ be a connected oriented closed $C^{\infty}$-manifold of
 dimension $m \ge 2$.
 We fixed a degree $p$ with $1 \le p \le m.$
 We take any smooth Riemannian metric $g_2$ on $M$ such that $g_2$ is flat 
 on the geodesic ball $B(x_2, 2)$ with the radius $2$ centered at $x_2 \in M$.

 For any $\eta >0$ and any index $k \geq 1$, from Theorem \ref{thm:sphere},
 there exists some $L_p >0$ such that for all $L > L_p$, 
%%%%%%%%
\begin{equation} \label{eq:small-eta-ev}
\begin{split}
  \overline{\lambda}^{(p)}_k ({\Bbb S}^m, \overline{g}_{p,L}) < \dfrac{\eta}{2}
   \quad  \text{ and } \quad 
  \lambda^{\prime \prime (p)}_k({\Bbb S}^m, \overline{g}_{p,L}) < \dfrac{\eta}{2},
\end{split}
\end{equation}
%%%%%%%
 where $\overline{g}_{p,L}$ is the volume normalized Riemannian metric 
 on ${\Bbb S}^m$ constructed in Theorem \ref{thm:sphere}. 
 By the continuity of the eigenvalues of the rough and Hodge Laplacians 
 acting on co-exact forms in the $C^0$-topology of Riemannian metrics
 (see \cite{Dodziuk[82]}, \cite[p.297]{Craioveanu-Puta-Rassias[01]}),
 after $C^0$-perturbation of $\overline{g}_{p,L}$,
 we may suppose that $\overline{g}_{p,L}$ is flat on a small geodesic ball
 $B(x_1, r_1)$  with the radius $r_1 >0$ centered at $x_1 \in {\Bbb S}^m$
 such that \eqref{eq:small-eta-ev} still holds.

 Now, we set $(M_1,g_1) := ({\Bbb S}^m, \overline{g}_{p,L})$ and
 $(M_2, g_2) : = (M,g).$
 By the construction of collapsing of the connected sum $(M, g_{\e})$
 from $(M_1,g_1)$ and $(M_2, g_2)$ as in Subsection $5.1$ \eqref{eq:conn-sums},
 we obtain a family of Riemannian metrics $g_{\e}$ on the connected sum 
 $M \cong {\Bbb S}^m \sharp M$ such that
%%%%%%%%
\begin{equation*} \label{eq:sphere-conn-sum}
\begin{split}
 (M, g_{\e}) &:= ({\Bbb S}^m(\e), \overline{g}_{p,L}) \cup_{\partial} (M(1), \e^2 g).
\end{split}
\end{equation*}
%%%%%%%

%%%%%%%%%%%%%%%%%%%%%%%%%%%%%%%%%%%%%
\begin{figure}[H] % [h]
  \centering
\begin{tikzpicture}[scale=0.6]
 % help lines
 %\draw [help lines] (-6.0,-4.0) grid (3.0,3.0);
 % sphere
  \draw (0,0) arc (10:349:2.5);
  \draw (0,0) .. controls (-0.25,-0.25) 
    and (-0.25,-0.75) .. (0,-0.9);
  \draw [dotted] (0,0) .. controls (0.25,-0.25) 
    and (0.25,-0.75) .. (0,-0.9);
  \node at (-2.5,-3.5) {$({\Bbb S}^m(\e), \overline{g}_{p,L})$};
 %% small manifold  M_2
  \draw (0,0) .. controls (0.2,-0.3) 
    and (0.5,-0.3) .. (0.7,-0.1);
  \draw (0.7,-0.1) .. controls (0.9,0.3) 
    and (0.5,0.5) .. (0.5,0.7);
  \draw (0.5,0.7) .. controls (0.5,1.1) 
    and (1.4,1.1) .. (1.5,0.7);
  \draw (1.5,0.7) .. controls (1.6,0.5) 
    and (1.3,0.3) .. (1.2,0.0);
  \draw (1.2,0.0) .. controls (1.1,-0.2) 
    and (1.2,-0.5) .. (1.2,-0.5);
 %%  symmetry
  \draw (1.2,-0.5) .. controls (1.3,-0.7) 
    and (1.5,-0.9) .. (1.6,-1.1);
  \draw (1.6,-1.1) .. controls (1.7,-1.3) 
    and (1.8,-1.6) .. (1.5,-1.7);
  \draw (0,-0.9) .. controls (0.2,-0.7) 
    and (0.5,-0.5) .. (0.7,-0.8);
  \draw (0.7,-0.8) .. controls (0.8,-1.0) 
    and (0.5,-1.2) .. (0.6,-1.5);
  \draw (0.6,-1.5) .. controls (0.6,-1.6) 
    and (0.8,-1.7) .. (1.0,-1.75);
  \draw (1.0,-1.75) .. controls (1.2,-1.8) 
    and (1.3,-1.8) .. (1.5,-1.7);
 %%  hole 1
  \draw (0.8,-1.2) .. controls (0.9,-1.4) 
    and (1.1,-1.5) .. (1.4,-1.25);
  \draw (0.85,-1.3) .. controls (0.9,-1.2) 
    and (1.1,-1.1) .. (1.35,-1.3);
 %%  hole 2
  \draw (0.7,0.7) .. controls (0.8,0.5) 
    and (1.1,0.3) .. (1.3,0.7);
  \draw (0.75,0.6) .. controls (0.9,0.7) 
    and (1.1,0.7) .. (1.25,0.6);
  \node at (1.2,-2.5) {$(M(1), \e^2 g)$};
\end{tikzpicture}
 \caption{$(M,g_{p,\varepsilon})$}
\end{figure}
%%%%%%%%%%%%%%%%%%%%%%%%%%%%%%%%%%%%%

 From Lemmas \ref{lem:p-conv} and \ref{lem:upper-bound}, 
 there exists some $\e_0 > 0$ such that for any $\e < \e_0$, 
%%%%%%%%
\begin{equation} \label{eq:ev-conv-eta}
\begin{split}
 0 \le \overline{\lambda}^{(p)}_k(M, g_{\e}) &\le
      \overline{\lambda}^{(p)}_k({\Bbb S}^m,\overline{g}_{p,L}) 
      + \dfrac{\eta}{2}, \\
 0 < \lambda^{\prime \prime (p)}_k (M, g_{\e}) &\le
      \lambda^{\prime \prime (p)}_k({\Bbb S}^m, \overline{g}_{p,L}) 
      + \dfrac{\eta}{2}.
\end{split}
\end{equation}
%%%%%%%
 By substituting \eqref{eq:small-eta-ev} to \eqref{eq:ev-conv-eta}, we obtain 
%%%%%%%%
\begin{equation*}
\begin{split}
 0 \le \overline{\lambda}^{(p)}_k(M, g_{\e}) &< \eta  \quad \text{ and } \quad 
       0 < \lambda^{\prime \prime (p)}_k (M, g_{\e}) < \eta.
\end{split}
\end{equation*}
%%%%%%%

 Finally, we normalize a Riemannian metric $g_{\e}$ on $M$. 
 If we set a new Riemannian metric 
%%%%%%%%
\begin{equation*}
\begin{split}
  \overline{g}_{\e} &:= \vol(M, g_{\e})^{-\frac{2}{m}} g_{\e}, 
\end{split}
\end{equation*}
%%%%%%%
 then $ \vol(M, \overline{g}_{\e}) \equiv 1.$
 From the volume convergence \eqref{eq:vol-conv}, there exists smaller
 $\e_1 \le \e_0$, if necessary, such that for all $\e < \e_1,$
%%%%%%%%
%\begin{equation*}
%\begin{split}
  $\vol(M, g_{\e}) \le  2.$
%\end{split}
%\end{equation*}
%%%%%%%
 Therefore, we have 
%%%%%%%%
\begin{equation*} 
\begin{split}
 \overline{\lambda}^{(p)}_k (M, \overline{g}_{\e})
  &=  \overline{\lambda}^{(p)}_k (M, g_{\e}) \cdot 
          \vol(M, g_{\e})^{\frac{2}{m}} 
   <  \eta \cdot 2^{\frac{2}{m}},
\end{split}
\end{equation*}
%%%%%%%
 and similarly  
 $\lambda^{\prime \prime (p)}_k (M, \overline{g}_{\e}) < \eta \cdot 2^{\frac{2}{m}}.$
 Since $\eta > 0$ is arbitrary, 
 we have finished the proof of Theorem \ref{thm:main-thm}.
\end{proof}

%%%%%%%%%%%%%%%%%%%%%%%%%%%%%%%%%%%%%%%%%%%%%%%%%%%%%%%%%%%%

\subsection{Proof of Theorem \ref{thm:unifomly-main}}

\begin{proof}[Proof of Theorem $\ref{thm:unifomly-main}$]
 We prove it in the same way as in the proof of Theorem \ref{thm:main-thm}.
 Let $M^m$ be a connected oriented closed $C^{\infty}$-manifold of
 dimension $m \ge 2$.
 We take a smooth Riemannian metric $g$ on $M$ such that $g$ is flat 
 on $(m-1)$ disjoint geodesic balls $B(x_i,2)$ for $i=1,2, \dots, m-1$ 
 with the radius $2$, centered at distinct $(m-1)$ points $x_i \in M$.
 By rescaling of $g$, we can do this.

 For any $\eta >0$ and any index $k \geq 1$, as in the proof of Theorem \ref{thm:main-thm},
 we can take a positive number $L_p >0$ safisfying that 
 the inequalities \eqref{eq:small-eta-ev} hold.
 For all $L > \displaystyle \max_{p=1,2,\dots,m-1} L_p$, 
 the inequalities \eqref{eq:small-eta-ev} hold uniformly for all degrees
 $p=1,2,\dots,m-1.$

 We consider now the $(m-1)$ spheres 
 $(\Bbb S^m, \overline{g}_{1,L})$, $(\Bbb S^m, \overline{g}_{2,L})$,
 $ \dots, (\Bbb S^m, \overline{g}_{m-1,L})$ and fix a point
 $x_0 \in {\Bbb S}^m$, it defines on each of them a point $x_{0,p}$.
 From the continuity of the eigenvalues of the rough and Hodge 
 Laplacians acting on co-exact forms in the $C^0$-topology of Riemannian
 metrics, we may suppose that each metric $\overline{g}_{p,L}$ on ${\Bbb S}^m$
 for $p=1,2,\dots,m-1$ is flat on each geodesic ball $B(x_{0,p}, r_1)$ 
 centered at $x_{0,p}$ with radius $r_1 >0$ small enough, 
 such that the inequalities \eqref{eq:small-eta-ev} still hold.

 We now perform the connected sum of $M$ and these $(m-1)$ spheres 
 $(\Bbb S^m, \overline{g}_{1,L})$, $(\Bbb S^m, \overline{g}_{2,L})$,
 $ \dots, (\Bbb S^m, \overline{g}_{m-1,L}),$ where $x_p\in M$ is related to
 $x_{0,p}\in \Bbb S^m$ equipped with the metric $\overline{g}_{p,L}$
 (See Figure $5$).
  We denote the resulting manifold by
%%%%%%%%
\begin{multline*} \label{eq:m-1-conn-sums}
 (M, g_{\e}) := \Big( ({\Bbb S}^m(\e), \overline{g}_{1,L}) \sqcup ({\Bbb S}^m(\e), \overline{g}_{2,L})
   \sqcup \dots \sqcup ({\Bbb S}^m(\e), \overline{g}_{m-1,L}) \Big) \\
    \cup_{\partial} (M \setminus \sqcup^{m-1}_{p=1} B(x_{p}, 1), \e^2 g),
\end{multline*}
%%%%%%%
 where ${\Bbb S}^m(\e) := {\Bbb S}^m \setminus B(x_0, \e)$.
 After smoothing the Riemannian metric $g$ on each neighborhood of
 $\partial (M \setminus B(x_{p},1))$ for $p=1,2,\dots,m-1$,
 we obtain a family of closed {\bf smooth} Riemannian manifolds $(M, g_{\e})$ for $\e>0$.

%%%%%%%%%%%%%%%%%%%%%%%%%%%%%%%%%%%%%
\begin{figure}[H]
  \centering
\begin{tikzpicture}[scale=0.8]
 % help lines
 % \draw [help lines] (-4.0,-4.0) grid (4.0,4.0);
 %% small manifold  M_2
 %%%%%
 \draw (-1.0,1.0) .. controls (-0.5,1.45) 
    and (0.5,1.45) .. (0.9,1.1);
 \draw (1.0,1.0) .. controls (1.5,0.7) 
    and (0.75,0.25) .. (0.8,0);
 \draw (0.8,0) .. controls (0.75,-0.25) 
    and (1.5,-0.25) .. (1.1,-0.9);
 \draw (-1.1,0.9) .. controls (-1.45,0.7) 
    and (-0.75,0.25) .. (-0.8,0);
 \draw (-0.8,0) .. controls (-0.75,-0.25) 
    and (-1.45,-0.25) .. (-1.0,-1.0);
 \draw (-0.9,-1.1) .. controls (-0.5,-1.45) 
    and (0.5,-1.45) .. (1.0,-1.0);
 %%%%%%%
 %%  hole
  \draw (-0.4,-0.1) .. controls (-0.2,0.2) 
    and (0.2,0.2) .. (0.4,-0.1);
  \draw (-0.5,0) .. controls (-0.2,-0.3) 
    and (0.2,-0.3) .. (0.5,0);
   \node at (2.2,0) {$(M, \e^2 g)$};
 %%%%%%%
 %  spheres
  \draw (1.0,1.0) arc (235:584:1.0); 
  \node[above] at (1.8,2.7) {$({\Bbb S}^m(\e), \overline{g}_{1,L})$};
  \draw (-1.0,1.0) arc (320:665:1.0);
   \node[above] at (-1.8,2.7) {$({\Bbb S}^m(\e), \overline{g}_{2,L})$};
  \draw (1.0,-1.0) arc (145:495:1.0);
   \node[below] at (1.8,-2.8) {$({\Bbb S}^m(\e), \overline{g}_{p,L})$};
  \draw (-1.0,-1.0) arc (50:405:1.0);
   \node[below] at (-1.8,-2.8) {$({\Bbb S}^m(\e), \overline{g}_{p-1,L})$};
%%%%%%
\end{tikzpicture}
 \caption{$(M,g_{\varepsilon})$}
\end{figure}
%%%%%%%%%%%%%%%%%%%%%%%%%%%%%%%%%%%%%

 For this $(M, g_{\e})$, we find that the same statement as in Lemma 
 \ref{lem:upper-bound} holds for the Hodge-Laplacian and rough Laplacian.
 In fact, we take the same cut-off function as in \eqref{eq:cut-off-func}
 for each component, and estimate the Rayleigh-Ritz quotients from above.
 Since a contribution from each cut-off function is within its own component,
 we obtain the same estimate as in \eqref{eq:upper-bdd}.
 
 Therefore, there exists some $\e_0 >0$ such that for any $\e < \e_0$ 
 and $p=1,2,\dots,m-1$ 
%%%%%%%%
\begin{equation} \label{eq:ev-conv-eta-2}
\begin{split}
 0 \le \overline{\lambda}^{(p)}_k(M, g_{\e}) &\le
      \overline{\lambda}^{(p)}_k({\Bbb S}^m, \overline{g}_{p,L}) 
      + \dfrac{\eta}{2}, \\
 0 < \lambda^{\prime \prime (p)}_k (M, g_{\e}) &\le
      \lambda^{\prime \prime (p)}_k({\Bbb S}^m, \overline{g}_{p,L}) 
      + \dfrac{\eta}{2}.
\end{split}
\end{equation}
%%%%%%%
 By substituting the same inequalities as in \eqref{eq:small-eta-ev} to 
 \eqref{eq:ev-conv-eta-2}, for small $\e >0$, we obtain 
%%%%%%%%
\begin{equation*}
\begin{split}
 0 \le \overline{\lambda}^{(p)}_k(M, g_{\e}) &< \eta  \quad \text{ and } \quad 
       0 < \lambda^{\prime \prime (p)}_k (M, g_{\e}) < \eta
\end{split}
\end{equation*}
%%%%%%%
 for all $p=1,2,\dots,m-1.$
 Since $\vol({\Bbb S}^m, \overline{g}_{p,L})$ is alomost $1$
 (because of small perturbation around one point), we find 
%%%%%%%%
\begin{equation*} 
\begin{split}
  \vol(M,g_{\e}) &\le \dsum^{m-1}_{p=1} \vol({\Bbb S}^m, \overline{g}_{p,L})
     + \vol(M, \e^2 g) \le m.
\end{split}
\end{equation*}
%%%%%%%
 After normalization of Riemannian metric $g_{\e}$ on $M$, we can find 
 a Riemannian metric $\overline{g}_{\e}$ on $M$ such that  
 $\vol (M, \overline{g}_{\e}) \equiv 1$ and 
%%%%%%%%
\begin{equation*} %\label{eq:ev-vol}
\begin{split}
 \overline{\lambda}^{(p)}_k (M, \overline{g}_{\e})
   &<  \eta \cdot m^{\frac{2}{m}}  \quad \text{ and } \quad 
 \lambda^{\prime \prime (p)}_k (M, \overline{g}_{\e}) < \eta \cdot m^{\frac{2}{m}}
\end{split}
\end{equation*}
%%%%%%%
 for all $p=1,2,\dots,m-1$.

 Thus, we have finished the proof of Theorem \ref{thm:unifomly-main}.
\end{proof}

%%%%%%%%%%%%%%%%%%%%%%%%%%%%%%%%%%%%%%%%%%%%%%%%%%%%%%%%%%%%%%%%%%%%%%%%%
%%%%%%%%               SECTION 
%%%%%%%%%%%%%%%%%%%%%%%%%%%%%%%%%%%%%%%%%%%%%%%%%%%%%%%%%%%%%%%%%%%%%%%%%

\section{Appendix: Convergence of the eigenvalues of the rough Laplacian}

 We study here the convergence of the eigenvalues of the rough Laplacian acting 
 on $p$-forms, when one side of a connected sum of two closed Riemannian manifolds
 collapses to a point.
 Our setting is the same as in the beginning of Subsection $5.1.$

%%%%%%%%%%%%%   theorem rough conv   %%%%%%%%%%%%%%%%%%%%%%%%
\begin{thm}\label{thm:rough-conv}
 For all $p$ with $0 \le p \le m$ and for all $k=1,2, \dots,$ we have
%%%%%%%%
\begin{equation*}
\begin{split}
  \lim_{\e \rightarrow 0} \overline{\lambda}^{(p)}_k(M,g_{\e}) &= 
  \overline{\lambda}^{(p)}_{k} (M_1,g_1).
\end{split}
\end{equation*}
\end{thm}
%%%%%%%%%%%%%%

To prove this Theorem \ref{thm:rough-conv}, we follow the schema of \cite{p-conv[12]}
which dealt with the Hodge Laplacian, but with less difficulties: when working with
the rough Laplacian the related quadratic form is exactly the one involved in the
$H^1$-norm and we are rather in the situation of \cite{0-conv[02]} which dealt with
functions. Nevertheless notice that we use here cut-off functions, while
 \cite{0-conv[02]} used a technique by means of the harmonic extension.

\vspace{0.3cm}
 Let $\vphi_{\e}$ be a normalized eigen $p$-form of the rough Laplacian 
 associated with the eigenvalue  
  $\overline{\lambda}^{(p)}_k({\e}) = \overline{\lambda}^{(p)}_k(M,g_{\e})$: 
%%%%%%
\begin{equation*}
\begin{split}
  \overline{\Delta} \vphi_{\e} 
    &=  \overline{\lambda}^{(p)}_k(\e) \vphi_{\e} \ 
  \text{ and }  \  
   \| \vphi_{\e} \|_{L^2(M, g_{\e})} \equiv 1.
\end{split}
\end{equation*}
%%%%%%%%
 By Lemma \ref{lem:upper-bound}, we already know that 
 the family $\{ \overline{\lambda}^{(p)}_k({\e}) \}_{\e>0}$ is bounded.  
 So, we set  
 $\overline{\lambda}^{(p)}_k = \displaystyle \liminf_{\e \rightarrow 0} 
   \overline{\lambda}^{(p)}_k (M,g_{\e}), $ 
 and decompose the eigen $p$-form $\vphi_{\e}$ on the connected sum into 
%%%%%%
\begin{equation*}
\begin{split}
  \vphi_{j,\e} &=  \big( \vphi^{1}_\e, \, \e^{p- \frac{m}{2}} \vphi^{2}_\e \big)
   \hbox{ with }
  \vphi^{1}_{\e}  \in H^1 (\Lambda^p M_1(\e), g_1), \
  \vphi^{2}_{\e}  \in H^1 (\Lambda^p M_2(1),  g_2).
\end{split}
\end{equation*}
%%%%%%%%
 Then, these satisfy
%%%%
\begin{equation} \label{eq:gluing-cond}
\begin{split}
  & \| \vphi^{1}_\e \|^2_{L^2(M_1(\e), g_1)} + 
      \| \vphi^{2}_\e \|^2_{L^2(M_2(1), g_2)} \equiv 1,  \\[0.2cm]
  & \vphi^{2}_\e = \e^{\frac{m}{2}-p} \vphi^{1}_{\e} 
        \hbox{ on the boundary.}
\end{split}
\end{equation}
%%%%%
 Furthermore, since $\vphi_{\e}$ is a normalized eigenform, we have 
%%%%%
\begin{equation} \label{eq:ev-quadratic}
\begin{split}
 \overline{\lambda}^{(p)}_{k}(M,g_{\e}) 
   &= \dint_{M_1(\e)} | \nabla \vphi^{1}_{\e} |^2 d \mu_{g_1} + 
      \dfrac{1}{\e^2} \dint_{M_2(1)} | \nabla \vphi^{2}_\e|^2 d \mu_{g_2} \\
   &= \| \nabla \vphi^{1}_{\e} \|^2_{L^2 (M_1(\e),g_1)} + 
      \dfrac{1}{\e^2} \| \nabla \vphi^{2}_{\e} \|^2_{L^2(M_2(1),g_2)}.
\end{split}
\end{equation}
%%%%%%%%%%
 From \eqref{eq:gluing-cond} and \eqref{eq:ev-quadratic}, it follows that 
%%%%%
\begin{equation*} 
\begin{split}
 \| \vphi^{2}_{\e} \|^2_{H^1(M_2(1),g_2)}
   &=  \| \vphi^{2}_{\e} \|^2_{L^2(M_2(1),g_2)}
     + \| \nabla \vphi^{2}_{\e} \|^2_{L^2(M_2(1),g_2)} \\
   &\le  1 + {\e}^2 \overline{\lambda}^{(p)}_{k}(M,g_{\e}).
\end{split}
\end{equation*}
%%%%%%%%%%
 By Lemma \ref{lem:upper-bound}, we see that 
 the family $\{ \vphi^{2}_{\e} \}_{\e >0}$ is bounded in 
 $H^1 (\Lambda^p M_2(1), g_2)$.
 Since $M_2(1)$ is compact, there exists a subsequence 
 $\{ \vphi^{2}_{\e_i} \}^{\infty}_{i=1}$ which converges weakly 
 to $\vphi^{2}$ in $H^1 (\Lambda^p M_2(1), g_2)$ and 
 strongly in $L^2 (\Lambda^p M_2(1), g_2)$.

%%%%%%%%%%%%%%%%
\begin{lem} \label{lem:phi^2-parallel}
 The sequence $\{ \vphi^{2}_{\e_i} \}$ converges strongly 
 to $\vphi^2$ in $H^1 (\Lambda^p M_2(1), g_2)$, and
 the limit $\vphi^{2}$ is parallel on $(M_2(1), g_2).$
\end{lem}
%%%%%%%%%%%

%%%%
\begin{proof}
 From the lower semi-continuity of the weak limit and 
 Lemma \ref{lem:upper-bound}, it follows that 
%%%%%
\begin{equation*} 
\begin{split}
 \| \nabla \vphi^{2} \|^2_{L^2(M_2(1),g_2)}
   &\le \liminf_{\e \to 0} \| \nabla \vphi^{2}_{\e} \|^2_{L^2(M_2(1),g_2)} \\
   &\le \liminf_{\e \to 0} {\e}^2 \overline{\lambda}^{(p)}_{k}(M,g_{\e})
    = 0,
\end{split}
\end{equation*}
%%%%%%%%%%
 that is, $\vphi^{2}$ is parallel on $(M_2(1), g_2)$.
 Therefore, we have 
%%%%%
\begin{equation*} 
\begin{split}
  \| \vphi^{2}_{\e_i} - \vphi^{2} \|^2_{H^1(M_2(1), g_2)}
  &= \| \vphi^{2}_{\e_i} - \vphi^{2} \|^2_{L^2(M_2(1), g_2)} + 
     \| \nabla \vphi^{2}_{\e_i} \|^2_{L^2(M_2(1), g_2)}  \\
  &\longrightarrow  0  \quad  (i \to \infty).
\end{split}
\end{equation*}
%%%%%%%%%%
\end{proof}
%%%%%%%%%%%

 The following boundary value estimate is crucial in our argument.

%%%%%%%%%%%%%%%  lem:boundary-value-estimate   %%%%%%%%
\begin{lem}\label{lem:boundary-value-estimate} 
 There exists a constant $C>0$ such that for any $r$ with $\e \le r \le r_1$,  
 $\vphi \in H^1(M_1(r), g_1)$ satisfies 
\begin{equation*}
\begin{split}
  \| \vphi \rest_{\partial M_1(r)} \, \|^2_{L^2 (\partial M_1(r), g_{1, \partial})}
    &\leq 
\begin{cases}
    \dfrac{C r}{m-2} \, \| \vphi \|^2_{H^1(M_1(r),g_1)} 
         &  \text{if } m \ge 3,  \\[0.3cm]
    C r \big| \log r \big| \, \| \vphi \|^2_{H^1(M_1(r),g_1)} 
         &  \text{if } m = 2.
\end{cases}
\end{split}
\end{equation*}
 Note that since $\vphi \in H^1(M_1(r),g_1),$ the boundary value 
 $\vphi \rest_{\partial M_1(r)}$ on $\partial M_1(r)$ is considered in the sense of
 the trace operator $H^1(\Lambda^p M_1(r),g_1) \longrightarrow
  L^2(\Lambda^p M_1(r),g_1)\rest_{\partial M_1(r)}.$
\end{lem}
%%%%%%%%%%%%%%%%%%

%%%%%%%%%%%%%%%%%%%
\begin{proof}
 We may assume that $\vphi$ is smooth.
 By using the polar coordinates $(r,\theta)$ on the geodesic ball $B(x_1,r_1),$
 we denote a $p$-form $\vphi = \a + dr \wedge \b$ and the metric $g_1 = dr^2 + r^2 h$, 
 where $h$ is the standard metric of ${\Bbb S}^{m-1}$.
 Then, the point-wise norm of $\vphi$ at $(r,\theta)$ is expressed as 
%%%%%
\begin{equation*} \label{eq:point-wise-norm}
\begin{split}
 |\vphi(r,\theta)|^2_{g_1}
    &= r^{-2p} |\a(r,\theta)|^2_h + r^{-2p +2} |\b(r,\theta)|^2_h.
\end{split}
\end{equation*}
%%%%%%%%
 We take a cut-off function $\chi$ on the  ball $B(x_1,r_1)$
 satisfying $\chi(s) = 1$ for $s\leq r$ and $\chi(r_1) = 0$
 (We may take $r_1 <1$, if necessary).
 From the Kato inequality $|\nabla |\vphi| | \le |\nabla \vphi|$
 and the Schwarz inequality, it follows that 
%%%%%%%%%%%%%%%%%%%%
\begin{equation*}
\begin{split}
 |\vphi (r,\theta)|_{g_1}
   &=  \dint^{r_1}_{r} \partial_s \big( |\chi \vphi (s,\theta)|_{g_1} \big) \, ds 
   \le  \dint^{r_1}_{r} |\nabla (\chi \vphi) (s,\theta)|_{g_1} \, ds  \\ 
   &\leq \sqrt{ \dint^{r_1}_r s^{1-m} \, ds} \cdot 
      \sqrt{ \dint^{r_1}_r |\nabla (\chi \vphi)(s,\theta)|^2_{g_1} s^{m-1} \, ds }.
\end{split}
\end{equation*}
%%%%
 Therefore, we have 
%%%%%%%%%%%%%%%%%%%%
\begin{equation*}
\begin{split}
 \| \vphi \rest_{\partial M_1(r)} \, \|^2_{L^2 (\partial M_1(r), g_{1, \partial})}
  &=  \dint_{\partial M_1(r)} |\vphi (r,\theta)|^2_{g_1} \, d \mu_{r^2 h}  
   =  r^{m-1} \dint_{{\Bbb S}^{m-1}} |\vphi (r,\theta)|^2_{g_1} \,  d \mu_{h} \\ 
  &\le r^{m-1} \, \dint^{r_1}_r s^{1-m} \, ds \cdot  \dint^{r_1}_r \dint_{{\Bbb S}^{m-1}}
         |\nabla (\chi \vphi)(s,\theta)|^2_{g_1} s^{m-1} \, ds \, d \mu_{h}  \\
  &= r^{m-1} \, \dint^{r_1}_r s^{1-m} \, ds \cdot \dint^{r_1}_r \dint_{{\Bbb S}^{m-1}}
       | \nabla \chi \otimes \vphi + \chi \nabla \vphi |^2_{g_1} 
      s^{m-1} \, ds \, d \mu_{h}  \\
  &\le C \, r^{m-1} \, \dint^{r_1}_r s^{1-m} \, ds \cdot
          \big\| \vphi \big\|^2_{H^1(M_1(r),g_1)}, 
\end{split}
\end{equation*}
%%%%
 where $C$ is a positive constant depending only on $\chi$ and $\nabla \chi$.
 By combining this with  
%%%%%%%%%%%%%
\begin{equation*}
\begin{split}
  \dint^{r_1}_r s^{1-m} \, ds &\leq  
\begin{cases}
    \dfrac{r^{2-m}}{m-2}  &  \text{if } m \ge 3,  \\[0.3cm]
    \big| \log r \big|  &  \text{if } m = 2, 
\end{cases}  
\end{split}
\end{equation*}
%%%%
 we obtain the boundary value estimate.
\end{proof}
%%%%%%%%%%%%%%

%%%%%%
\begin{lem}\label{lem:phi^2=0}
 The limit $\vphi^{2} = 0$ a.e.\ $(M_2(1), g_2).$
\end{lem}
%%%%%%%%

%%%%%%%%
\begin{proof}
 Since $\vphi^{2}$ is parallel, it is sufficient to prove that the boundary 
 value of $\vphi^{2}$ to $\partial M_2(1)$ in the sense of the trace is zero.
 Since the trace operator is continuous and 
 $\vphi^{2}_{\e_i} \longrightarrow \vphi^{2}$ strongly in  $H^1 (M_2(1), g_2)$, 
 we have
%%%%%%%%%
\begin{equation*}
\begin{split}
 \| \vphi^{2}_{\e_i} \rest_{\partial M_2(1)} - \vphi^{2} \rest_{\partial M_2(1)}
    \|^2_{L^2 (\partial M_2(1), g_{2, \partial})} 
  &\le C \| \vphi^{2}_{\e_i} - \vphi^{2} \|^2_{H^1(M_2(1), g_{2})}
   \rightarrow 0  \quad (i \to \infty).
\end{split}
\end{equation*}
%%%%%%%%
 Thus, we see the norm convergence:
%%%%%%%%%
\begin{equation*}
\begin{split}
 \| \vphi^{2} \rest_{\partial M_2(1)} \|_{L^2 (\partial M_2(1), g_{2, \partial})} 
  = \dlim_{i \to \infty} \| \vphi^{2}_{\e_i} \rest_{\partial M_2(1)}
    \|_{L^2 (\partial M_2(1), g_{2, \partial})}.
\end{split}
\end{equation*}
%%%%%%%%

 From the gluing condition \eqref{eq:gluing-cond} at the boundary,
 we have 
%%%%%%%
\begin{equation*} 
\begin{split}
 \| \vphi^2_{\e_i} \rest_{\partial M_2(1)} \|^2_{L^2 (\partial M_2(1), g_{2, \partial})}
  &= \dint_{\partial M_2(1)} | \vphi^2_{\e_i} \rest_{\partial M_2(1)} |^2_{g_2}
        \, d \mu_{h} \\
  &= \dint_{\partial M_1(1)} | {\e_i}^{\frac{m}{2} - p}  \vphi^{1}_{\e_i}
        \rest_{\partial M_1(\e_i)} |^2_{g_2} \, d \mu_{h} 
      \quad  (\text{by } \eqref{eq:gluing-cond}) \\
  &= {\e_i} \dint_{\partial M_1(\e_i)} {\e_i}^{- 2p}
      | \vphi^{1}_{\e_i} \rest_{\partial M_1(\e_i)} |^2_{g_2} \, d \mu_{{\e_i}^2 h} \\
  &= {\e_i} \dint_{\partial M_1(\e_i)} | \vphi^{1}_{\e_i} \rest_{\partial M_1(\e_i)} 
       |^2_{g_1} \, d \mu_{{\e_i}^2 h} \\ 
  &= {\e_i} \, \| \vphi^{1}_{\e_i} \rest_{\partial M_1(\e_i)} \,
       \|^2_{L^2 (\partial M_1(\e_i), g_{1, \partial})}.
\end{split}
\end{equation*}
 By Lemma \ref{lem:boundary-value-estimate}, 
 we find that the boundary value of $\vphi^{2}$ is zero.
 Therefore, the limit $\vphi^{2}$ must be zero.
\end{proof}
%%%%%%%%%%%%%%

 We take again the cut-off function $\chi_{\e}$ on $M_1$
 as in \eqref{eq:cut-off-func}, and set 
%%%%%%%%%%%%%
\begin{equation} \label{eq:psi}
\begin{split}
  \psi_{\e} &:= \chi_{\e} \vphi_{\e}^{1} \ \text{ on } M_1.
\end{split}
\end{equation}
%%%%

%%%%%%
\begin{lem}\label{lem:H^1-bdd}
 The family $\{ \psi_{\e} \}_{\e>0}$ is bounded in $H^1(\Lambda^p M_1, g_1)$.
\end{lem}
%%%%%%%%

\begin{proof}
  It is easy to see that the $L^2$-norm of $\psi_{\e}$ is bounded by $1$. 
  Now, we have
%%%%
\begin{equation*}
\begin{split}
 \dint_{M_1} & | \nabla (\chi_\e \vphi^1_\e) |^2_{g_1} d \mu_{g_1}
   = \dint_{M_1} | \nabla \chi_{\e} \otimes \vphi^1_{\e} +
         \chi_\e \nabla \vphi^1_{\e} |^2_{g_1} \, d \mu_{g_1} \\
   &\leq 2  \Big( \dfrac{2}{\log\e} \Big)^2 \,  \dint_\e^{\sqrt{\e}}
        \dint_{\Bbb S^{m-1}} | \vphi^1_\e (r,\theta) |^2_{g_1} \, r^{m-3} \, 
        dr \,d \mu_{h} 
      + 2 \dint_{M_1(\e)} |\nabla \vphi^1_{\e}|^2_{g_1} \, d \mu_{g_1} \\
   &=  \dfrac{8}{| \log \e |^2} \,  \dint_\e^{\sqrt{\e}}
        \| \vphi^1_{\e} \|^2_{L^2( \partial M_1(r), g_{1, \partial})} \,
        r^{-2} \, dr  
      + 2 \, \| \nabla \vphi^1_\e \|^2_{L^2(M_1(\e), g_1)}.
\end{split}
\end{equation*}
 For the first term, by applying Lemma \ref{lem:boundary-value-estimate},
 we have, if $m \ge 3$, 
\begin{equation*}
\begin{split}
 \dfrac{8}{| \log \e |^2} \, \dint^{\sqrt{\e}}_{\e}
     \| \vphi^{1}_{\e} \|^2_{L^2( \partial M_1(r), g_{1, \partial})} \, r^{-2} \, dr
  &\le \dfrac{C}{m-2} \cdot \dfrac{8}{| \log \e |^2} \,
     \| \vphi^{1}_{\e} \|^2_{H^1(M_1(r),g_1)} \, \dint^{\sqrt{\e}}_{\e} r^{-1} \, dr \\
  &\le  \dfrac{C}{m-2} \cdot \dfrac{4}{| \log \e |} \,
     \| \vphi^{1}_{\e} \|^2_{H^1(M_1(r),g_1)},
\end{split}
\end{equation*}
 and if $m=2$,  
\begin{equation*}
\begin{split}
 \dfrac{8}{| \log \e |^2} \, \dint^{\sqrt{\e}}_{\e}
     \| \vphi^1_{\e} \|^2_{L^2( \partial M_1(r), g_{1, \partial})} \, r^{-2} \, dr
  &\le  \dfrac{8C}{| \log \e |^2} \| \vphi^{1}_{\e} \|^2_{H^1(M_1(r),g_1)} \, 
        \dint^{\sqrt{\e}}_{\e} | \log r | \, r^{-1} \, dr \\
  &\le   \dfrac{8C}{| \log \e |} \| \vphi^{1}_{\e} \|^2_{H^1(M_1(r),g_1)} \, 
        \dint^{\sqrt{\e}}_{\e} r^{-1} dr  \\
  &=  4C \, \| \vphi^{1}_{\e} \|^2_{H^1(M_1(r),g_1)}.
\end{split}
\end{equation*}
%%%%
 For the second term, we have 
\begin{equation*}
\begin{split}
 \| \nabla \vphi_{\e} \|^2_{L^2(M_1(\e), g_1)}
  &\le  \| \nabla \vphi_{\e} \|^2_{L^2(M, g_{\e})}  
   =  \overline{\lambda}^{(p)}_k(M, g_{\e}), 
\end{split}
\end{equation*}
 which is uniformly bounded by Lemma \ref{lem:upper-bound}. 

 Therefore, we find that $\{ \psi_{\e} \}_{\e>0}$ is bounded
 in $H^1(\Lambda^p M_1, g_1)$.
\end{proof}
%%%%%%%%%%%%%%%%%%

 The following lemma is obtained from the same method as in
 \cite[Corollary $15$]{p-conv[12]}, p.1732.

%%%%%%%%%%%%%%%%%%%%%%%
\begin{lem}\label{cor:psi-conv}
 We can extract a subsequence $\{ \psi_{\e_i} \}$ which converges 
 weakly to $\psi$ in $H^1(\Lambda^p M_1, g_1)$ and
 strongly in $L^2(\Lambda^p M_1, g_1)$ such that
\begin{equation*}
 \begin{split}
  \overline{\Delta}_{M_1} \psi = \overline{\lambda}^{(p)}_k \psi \ \text{ and } 
     \  \| \psi \|_{L^2(M_1, g_1)} =1.  
\end{split}
\end{equation*}
\end{lem}
%%%%%%%%%%%%%%%%%%%%%%%%%%%%%%%%%

%%%%%%%%%%%%%%%%%%%%%%%%%%%%%
\begin{proof}
 From Lemma \ref{lem:H^1-bdd}, a family $\{ \psi_{\e} \}$  is uniformly bounded
 in $H^1( \Lambda^p M_1, g_1)$.
 By the weak compactness for a Hilbert space and the Rellich-Kondrachov theorem, 
 there exist a subsequence $\{ \psi_{\e_i} \}_{i}$ 
 and the limit $\psi \in H^1( \Lambda^p M_1, g_1)$ such that  
%%%%%%%%%%%%%%%%%%
  $\psi_{\e_i}  \rightarrow  \psi$   
 weakly in $H^1(\Lambda^p M_1, g_1)$ and strongly in $L^2(\Lambda^p M_1, g_1)$
 as $i \to \infty$.

 For any smooth $p$-form $\omega \in \Omega^p_0 (M_1 \setminus \{ x_1 \} ),$ 
 there exists $\e_0 > 0$ such that the support of $\omega$ is in 
 $M_1 \setminus B(x_1,2 \sqrt{\e_0} )$.
 So on this support we have $\psi_{\e_i}=\vphi^1_{\e_i}=\vphi_{\e_i}$ 
 as far as $\e_i<\e_0$. We label with $(\star)$ when we use this fact. 
 By Lemma \ref{lem:phi^2=0}, we have 
%%%%%%%%%%%%%%%%%%
\begin{equation*}
\begin{split}
 (\psi, \overline{\Delta}_{g_1} \omega)_{L^2(M_1, g_1)} 
   &= \lim_{i \to \infty} ( \psi_{\e_i}, \overline{\Delta}_{g_1} \omega 
         )_{ L^2(M_1, g_1) }    \\
   &\underset{(\star)}{=} \lim_{i \to \infty} 
         ( \vphi_{\e_i}, \overline{\Delta}_{g_{\e_i}} \omega)_{L^2(M_1({\e_i}),g_1)} \\
   &= \lim_{i \to \infty} 
         ( \vphi_{\e_i}, \overline{\Delta}_{g_{\e_i}} \omega)_{L^2(M, g_{\e_i})} 
      \quad (\text{by Lemma $\ref{lem:phi^2=0}$}) \\
   &= \lim_{i \to \infty} \overline{\lambda}^{(p)}_k (M, g_{\e_i}) (\vphi_{\e_i}, 
             \omega )_{L^2 (M, g_{\e_i}) } 
    = \overline{\lambda}^{(p)}_k \, \lim_{i \to \infty} (\vphi^1_{\e_i}, \omega 
          )_{L^2 (M_1(\e_i), g_1)} \\
   &\underset{(\star)}{=}  \overline{\lambda}^{(p)}_k \, \lim_{i \to \infty}
       ( \psi_{\e_i}, \omega)_{L^2(M_1,g_1)}
    = \overline{\lambda}^{(p)}_k \, (\psi, \omega)_{L^2(M_1,g_1)}.
\end{split}
\end{equation*}

%%%%%%%%%%%%%%%%%% 

%%%%%%%%%%%%%%%%%%
 Since $m \ge 2$, 
 $\Omega^{p}_0(M_1 \setminus \{ x_1 \})$ is dense in $H^1 (\Lambda^p M_1, g_1),$ 
 and we conclude that 
%%%%%%%%%%%%%%%%%%
\begin{equation*}
\begin{split}
  \overline{\Delta}_{g_1} \psi &= \overline{\lambda}^{(p)}_k \psi
     \quad  \text{ weakly.}
\end{split}
\end{equation*}
%%%%%%%%%%%%%%%%%%
  Furthermore, by the regularity theorem of weak solutions to elliptic equations,
 the limit $\psi$ in fact is a smooth $p$-form on $M_1$.

 Next, from the normalization $\| \vphi_{\e_i} \|_{L^2(M, g_{\e_i})} \equiv 1$ 
 and Lemma \ref{lem:phi^2=0}, we have 
  $\| \psi \|_{L^2(M_1,g_1)} =1.$
 Hence, the limit $\psi$ is a non-zero smooth eigenform on $(M_1, g_1)$ with 
 the eigenvalue $\overline{\lambda}^{(p)}_k.$
\end{proof}
%%%%%%%%%%%%%%%%%

 To complete the proof of Theorem \ref{thm:rough-conv}, 
 we have only to prove the following lemma.

%%%%%%
\begin{lem}\label{lem:limit-onf}
 Let $\{ \vphi_{1, \e_i}, \dots, \vphi_{k,\e_i} \}$ be $L^2(M,g_{\e_i})$-orthonormal
 eigenforms  on $(M,g_{\e})$ associated with the eigenvalues 
 $\overline{\lambda}^{(p)}_1(M, g_{\e_i}), \dots, \overline{\lambda}^{(p)}_k(M,g_{\e_i}),$
 and let $\{ \psi_{1}, \dots, \psi_{k} \}$ be the limits obtained from
  $\{ \vphi_{1, \e_i}, \dots, \vphi_{k,\e_i} \}.$
 Then,  $\{ \psi_{1}, \dots, \psi_{k} \}$ are also $L^2(M_1,g_1)$-orthonormal
 eigenforms on $(M_1,g_1)$ associated with the eigenvalues 
 $\overline{\lambda}^{(p)}_1(M_1, g_1), \dots$, $ \overline{\lambda}^{(p)}_k(M_1,g_1).$
\end{lem}
%%%%%%%%

%%%%
\begin{proof}We first calculate:
\begin{equation} \label{eq:chi-1=0}
\begin{split}
 \| (\chi_{\e_i} -1) \vphi_{\e_i} \|^2_{L^2 (M_1(\e),g_1) }
  &\le \dint^{\sqrt{\e_i}}_{\e_i} \dint_{{\Bbb{S}^{m-1}}}
         |\vphi_{\e_i} |^2_{g} \, r^{m-1} dr d \mu_{h}   \\[0.2cm]
  &\le  C
\begin{cases}
  \dfrac{1}{m-2} \, \dint^{\sqrt{\e_i}}_{\e_i} r \, dr \, 
      \| \vphi_{\e_i} \|^2_{H^1(M_1(\e_i), g_1)}  &   \text{if } m \ge 3, \\
  |\log \e_i| \, \dint^{\sqrt{\e_i}}_{\e_i} r \, dr \, 
      \| \vphi_{\e_i} \|^2_{H^1(M_1(\e_i), g_1)}  &   \text{if } m=2
\end{cases}   \\[0.2cm]
  &\le C
\begin{cases}
  \dfrac{\e_i}{m-2} \, \| \vphi_{\e_i} \|^2_{H^1(M_1(\e_i), g_1)}
        &   \text{if } m \ge 3, \\
  |\log \e_i| \e_i  \, \| \vphi_{\e_i} \|^2_{H^1(M_1(\e_i), g_1)}
        &   \text{if } m=2
\end{cases}  \\[0.2cm]
  &\longrightarrow 0  \quad (i \to \infty).
\end{split}
\end{equation}
Then, from $\dlim_{i \to \infty} \vphi^2_{j, \e_i} = 0$ by Lemma \ref{lem:phi^2=0} 
 and \eqref{eq:chi-1=0}, it follows that for all $j,l=1,\dots,k,$
%%%%%%%%%%%%%%%%%%
\begin{equation*} 
\begin{split}
 (\psi_j, \psi_l)_{L^2(M_1, g_1)} 
   &= \lim_{i \to \infty} ( \chi_{\e_i} \vphi_{j, \e_i}, \chi_{\e_i} \vphi_{l, \e_i}
        )_{ L^2(M_1, g_1) }  \\
   &= \lim_{i \to \infty} \Big\{ (\vphi^1_{j, \e_i}, \vphi^1_{l, \e_i}
        )_{ L^2(M_1(\e_i), g_1) } +  ( (\chi^2_{\e_i} -1) \vphi^1_{j, \e_i},
       \vphi^1_{l, \e_i})_{ L^2(M_1(\e_i), g_1) } \Big\}  \\
   &= \lim_{i \to \infty} \Big\{ (\vphi^1_{j, \e_i}, \vphi^1_{l, \e_i}
        )_{ L^2(M_1(\e_i), g_1) } + (\vphi^2_{j, \e_i}, \vphi^2_{l, \e_i}
        )_{ L^2(M_2(1), g_2) }  \Big\}  \\
   &\qquad  + \lim_{i \to \infty} ( (\chi^2_{\e_i} -1) \vphi^1_{j, \e_i},
       \vphi^1_{l, \e_i})_{ L^2(M_1(\e_i), g_1) } \\
   &=  \lim_{i \to \infty} (\vphi_{j, \e_i}, \vphi_{l, \e_i} )_{ L^2(M, g_{\e_i})}
       + \lim_{i \to \infty} ( (\chi^2_{\e_i} -1) \vphi^1_{j, \e_i},
       \vphi^1_{l, \e_i})_{ L^2(M_1(\e_i), g_1) }  \\
   &=  \delta_{j l}.
\end{split}
\end{equation*}
%%%%%%%%%%%%%%%%%%
 Here, the last equality follows from \eqref{eq:chi-1=0}.
 Therefore, we conclude that 
 $\overline{\lambda}^{(p)}_j = \dlim_{i \to \infty} \overline{\lambda}^{(p)}_j (M, g_{\e_i})$
 for $j=1,\dots, k$ belong to the set of all eigenvalues of the rough Laplacian 
 acting on $p$-forms on $(M_1, g_1)$.
 Hence, we have finished the proof of Theorem \ref{thm:rough-conv}.
\end{proof}
%%%%%%%%%%%%%%%%%

%%%%%%%%%%%%%%%%%%%%%%%%%%%%%%%%%%%%%%%%%%%%%%%%%%%%%%%%%%%%%%%%%%%%%%%%%%%%%
%%%%%%             Bibliography                                         %%%%%
%%%%%%%%%%%%%%%%%%%%%%%%%%%%%%%%%%%%%%%%%%%%%%%%%%%%%%%%%%%%%%%%%%%%%%%%%%%%%

%%%%%%%%%%%%%%%%%%%%%%%%%%%%%%

 \vspace{0.5cm}
%%%%%%%%%%%%
\noindent
 \ Colette Ann\'e \\
 \quad  Laboratoire de Math\'ematiques Jean Leray, \\
 \quad  Universit\'e de Nantes, CNRS, Facult\'e des Sciences, \\
 \quad  BP 92208, 44322 Nantes, France. \\
 \quad  colette.anne@univ-nantes.fr

\vspace{0.5cm}
%%%%%%%%%%%%
\noindent
 \ Junya Takahashi \\
 \quad  Research Center for Pure and Applied Mathematics, \\
 \quad  Graduate School of Information Sciences, \\
 \quad  T\^{o}hoku University, \\
 \quad  $6$--$3$--$09$ Aoba, Sendai $980-8579$, Japan. \\
 \quad  e-mail: t-junya@tohoku.ac.jp

\end{document}